\documentclass[reqno]{amsart}
\usepackage[utf8]{inputenc}
\usepackage{amsmath,amsthm,amssymb}
\usepackage{hyperref}
\date{\today}
\usepackage{paralist}

\theoremstyle{definition}
\newtheorem{theorem}{Theorem}[section]
\newtheorem*{maintheorem}{Main Theorem}
\newtheorem{theorem*}{Theorem}
\newtheorem{remark}[theorem]{Remark}
\newtheorem{lemma}[theorem]{Lemma}
\newtheorem{corollary}[theorem]{Corollary}
\newtheorem{example}[theorem]{Example}
\newtheorem{proposition}[theorem]{Proposition}
\newcommand{\Herm}{\operatorname{Herm}}

\theoremstyle{definition}
\theoremstyle{definition}
\newtheorem{definition}[theorem]{Definition}
\newtheorem*{definition*}{Definition}

\usepackage{hyperref}
\usepackage{graphicx}
\usepackage[textsize=tiny]{todonotes}
\usepackage{comment}
\begin{document}

\newcommand{\Pos}{Pos}
\renewcommand{\Im}{\operatorname{Im}}

\newcommand{\Hess}{\operatorname{Hess}}
\newcommand{\Id}{\operatorname{Id}}
\newcommand{\Sym}{\operatorname{Sym}}
\newcommand{\Hom}{\operatorname{Hom}}

\newcommand{\argmin}{\operatorname{argmin}}
\newcommand{\diag}{\operatorname{diag}}

\renewcommand{\Pos}{\operatorname{Pos}}
\newcommand{\calPos}{\mathcal P}
\newcommand{\Int}{\operatorname{Int}}
\setcounter{tocdepth}{1}

\title[The Minimum Principle for Convex Subequations]{The Minimum Principle for Convex Subequations} 
\author{Julius  Ross and David Witt Nystr\"om}

\begin{abstract}
A subequation, in the sense of Harvey-Lawson, on an open subset $X\subset \mathbb R^n$ is a subset $F$ of the space of $2$-jets on $X$ with certain properties.   A smooth function is said to be $F$-subharmonic if all of its $2$-jets lie in $F$, and using the viscosity technique  one can extend the notion of $F$-subharmonicity to any upper-semicontinuous function.    Let $\mathcal P$ denote the subequation consisting of those $2$-jets whose Hessian part is semipositive.   

We introduce a notion of  product subequation $F\#\calPos$ on $X\times \mathbb R^{m}$ and prove, under suitable hypotheses, that if $F$ is convex and $f(x,y)$ is $F\#\calPos$-subharmonic then the marginal function
$$ g(x):= \inf_y f(x,y)$$ is $F$-subharmonic.  

This generalises the classical statement that the marginal function of a convex function is again convex.  We also prove a complex version of this result that generalises the Kiselman minimum principle for the marginal function of a plurisubharmonic function.  
\end{abstract}
\maketitle

\tableofcontents

\section{Introduction}\label{sec:introduction}
Although the maximum of two convex functions is always convex, the same is not normally true for the minimum.  However there is a minimum principle for convex functions that states that if $f(x,y)$ is a function that is convex in two real variables then its marginal function
$$ g(x): = \inf_{y} f(x,y)$$
is again convex.   This fundamental property is used throughout the study of convex functions, in particular when considering convex optimisation problems.    In fact, since a function is convex if and only if its epigraph is a convex set, this minimum principle can be viewed as precisely the statement that the linear projection of a convex set is again convex.

In the complex case, convexity is replaced with the property of being plurisubhamonic, and then the statement becomes that if $f(z,w)$ is a plurisubharmonic function of two complex variables that is independent of the argument of $w$ then the marginal function $g(z) = \inf_{w} f(z,w)$ is again plurisubharmonic.  This minimum principle is due to Kiselman, and is a key tool in pluripotential theory.

\begin{center}
*
\end{center}

Darvas-Rubinstein \cite{RubinsteinDarvas} were the first to show that this minimum principle extends beyond the above classical setting asked if it holds even more generally.       The natural setting for this question is a huge generalization of convexity and plurisubharmonicity that uses the technique of viscosity subsolutions, as expounded by Harvey-Lawson  \cite{HL_Dirichletduality,HL_Dirichletdualitymanifolds,HL_equivalenceviscosity}.  

To explain this elegant idea, consider the real case of convex functions (the complex case being completely analogous).    If $g:X\to \mathbb R$ is a smooth function on an open $X\subset\mathbb R^n$, local convexity of $g$ is equivalent to the statement that for all $x\in X$ the Hessian $\Hess_{x}(g)$ is contained in the set of semipositive matrices.  If $g:X\to \mathbb R\cup \{-\infty\}$ is merely upper-semicontinuous, then being locally convex is equivalent to the statement  that for any smooth ``test-function" $\phi$ that touches $g$ from above at $x\in X$  (Figure \ref{fig:viscosity}) it holds that $\Hess_{x}\phi$ is semipositive.   

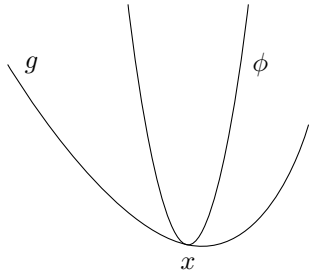
\begin{figure}[h]
\begin{tikzpicture}[scale=0.8]
%\draw [gray!50]  (0,0) -- (1,1) -- (3,1) -- (1,0)  -- (2,-1) -- cycle;
%\draw [red] plot [smooth cycle] coordinates {(0,0) (1,1) (3,1) (1,0) (2,-1)};
%\draw [gray!50, xshift=4cm]  (0,0) -- (1,1) -- (2,-2) -- (3,0);
\draw [xshift=4cm] plot [smooth, tension=1] coordinates { (1,1) (4,-2) (6,0)};
\draw [xshift=4cm] plot [smooth, tension=1] coordinates { (3,2) (4,-2) (5,2)};
\node  at (5.4,1.0) {$g$};
\node  at (9.2,1) {$\phi$};
\node  at (8,-2.3) {$x$};
\end{tikzpicture}
  \caption{The function $\phi$ touching $g$ from above at $x$}
\label{fig:viscosity}
\end{figure}

Now, we are free to replace the cone of semipositive matrices with any other subset $F$ of symmetric matrices, and in doing we can define what it means to be \emph{$F$-subharmonic} in precisely the same way.    In fact if instead of using the Hessian we instead use the full $2$-jet, we can take $F$ to be any subset of the space of $2$-jets.    To have a useful theory we need to make some mild assumptions on $F$.

\begin{definition*}
Let $X\subset \mathbb R^n$ be open and $F$ be a subset of the space $$J^2(X) = X\times \mathbb R\times \mathbb R^n\times \Sym_n^2$$ of $2$-jets on $X$. We say that $F$ is a \emph{primitive subequation} if 
\begin{enumerate}
\item $F$ is closed.
\item If $(x,r,p,A)\in F$ and $P$ is semipositive then $(x,r,p,A+P)\in F$.
\end{enumerate}
\end{definition*}

Somewhat surprisingly, even at this level of generality the space of $F$-subharmonic functions has many properties in common with convex and plurisubharmonic functions.  In this paper we define and prove a minimum principle in this setting.

\begin{center}
*
\end{center}

For a precise statement, suppose that $F\subset J^2(X)$ is a primitive subequation and let
 $$\mathcal P = \{ (x,r,p,A) \in J^2(\mathbb R^m) : A \text{ is semipositive} \}$$
 be the set of $2$-jets whose Hessian part is semipositive.    We will define a new primitive subequation 
 $$F\#\calPos\subset J^2(X\times \mathbb R^m)$$ 
 in such as way that
an $F\#\calPos$-subharmonic function is an upper-semicontinuous $$f:X\times \mathbb R^m\to\mathbb R\cup \{-\infty\}$$ whose restriction to any non-vertical slice is $F$-subharmonic, and whose restriction to any vertical slice is $\calPos$-subharmonic.  That is, $f$ is $F\#\calPos$-subharmonic if and only if
\begin{enumerate}[(i)]
\item For each linear $\Gamma:\mathbb R^{n}\to \mathbb R^{m}$ and $y_0\in \mathbb R^{m}$ the function
$$ x\mapsto f(x, y_0 +\Gamma x)$$
is $F$-subharmonic, and
\item For each fixed $x_0$ the function $y\mapsto f(x_0,y)$ is $\calPos$-subharmonic (i.e.\ locally convex).
\end{enumerate}

We say that a subset $\Omega\subset X$ is \emph{$F$-pseudoconvex} if $\Omega$ admits a continuous and exhaustive $F$-subharmonic function.

\begin{definition*}\label{def:minimumprinciple_intro} Let  $\pi:X\times \mathbb R^m\to X$ denote the natural projection.  
We say that a primitive subequation $F\subset J^2(X)$ satisfies the \emph{minimum principle} if the following holds:
 
Suppose $\Omega\subset X\times \mathbb R^{m}$ is an $F\#\calPos$-pseudoconvex domain such that that the slices 
$$\Omega_x : = \{ y \in \mathbb R^m : (x,y)\in \Omega\}$$ are connected for each $x\in X$.  Then $\pi(\Omega)$ is $F$-pseudoconvex, and for any $F\#\calPos$-subharmonic function $f$ on $\Omega$, the marginal function 
 $$g(x) : = \inf_{y\in \Omega_x} f(x,y)$$
is $F$-subharmonic on $\pi(\Omega)$.
\end{definition*}

In the complex case we take $X\subset \mathbb C^n$, and $F\subset J^{2,\mathbb C}(X)$ a complex primitive subequation (i.e.\ a subset of the space of complex $2$-jets on $X$ with the same properties as above).   Denote by $\mathcal P^{\mathbb C}\subset J^{2,\mathbb C}(\mathbb C^m)$ the set of complex $2$-jets whose complex Hessian is semipositive.    We then make an analogous definition of product $F\#_{\mathbb C}\calPos^{\mathbb C}$ in which we require the $\Gamma:\mathbb C^n
\to \mathbb C^m$ to be $\mathbb C$-linear.   The complex version of the minimum principle is similar, but  we require both  $\Omega\subset X\times \mathbb C^m$ and $f$  to be independent of the argument of the second variable.

\begin{definition*}\label{def:minimumprinciple_complex_intro}\ Let $\pi:X\times \mathbb C^m\to X$ denote the natural projection.  We say that a complex primitive subequation $F\subset J^{2,\mathbb C}(X)$ satisfies the \emph{minimum principle} if the following holds:
 
Suppose $\Omega\subset X\times \mathbb R^{m}$ is an $F\#\calPos^\mathbb C$-pseudoconvex domain such that that the slices $\Omega_z$ are connected for each $z\in X$ and are independent of the argument of the second variable.   Then $\pi(\Omega)$ is $F$-pseudoconvex, and for any $F\#\calPos^\mathbb C$-subharmonic function $f$ on $\Omega$ that is independent of the argument of the second variable, the marginal function 
 $$g(w) : = \inf_{w\in \Omega_z} f(z,w)$$
is $F$-subharmonic on $\pi(\Omega)$.
\end{definition*}

We conjecture that the minimum principle holds for a wide-class of primitive subequations $F$, and in this paper we prove this in complete generality for constant-coefficient convex primitive equations that have one additional (mild) property.

\begin{definition*}
We say that $F\subset J^2(X)$ has the \emph{Negativity Property} if 
$$(x,r,p,A)\in F\text{ and }r'<r \Rightarrow (x,r',p,A)\in F.$$
We say that $F$ is \emph{constant coefficient} if  for all $x,x'\in X$ we have
$$ (x,r,p,A)\in F \Leftrightarrow (x',r,p,A)\in F,$$
and $F$ is \emph{convex} if the fibre $F_x= \{ (r,p,A) :  (x,r,p,A)\in F\}$ is convex for each $x\in X$.
\end{definition*}
%\subsection{The minimum principle in the convex case}

\begin{maintheorem}[Minimum Principle in the Convex Case]\label{thm:minimumprincipleconvexI:intro}
 Let $X\subset \mathbb R^n$ be open and $F\subset J^2(X)$ be a real or complex primitive subequation such that
 \begin{enumerate}
% \item $\mathbb R_{\le 0} \oplus 0 \oplus 0 \subset F_x$ for all $x\in X$
\item $F$ satisfies the Negativity Property,
 \item $F$ is convex,
 \item $F$ is constant coefficient.
% \item $F$ and $G$ are independent of the gradient part.
 %\item The constant functions on $Y$ are $G$-subharmonic (so $G$ is a convex cone)
 %\item $G\subset \calPos$.
 \end{enumerate}
Then $F$ satisfies the minimum principle.
 \end{maintheorem}
 
It is not hard to check that $\calPos_{\mathbb R^n}\#\calPos_{\mathbb R^m} = \calPos_{\mathbb R^{n+m}}$ (and similarly in the complex case).    So when $F=\calPos_{\mathbb R^n}$,  the above Theorem reduces to the classical statement that the marginal function of a convex function is again convex.   Similarly when $F=\calPos^{\mathbb C}_{\mathbb C^n}$ we get precisely the Kiselman minimum principle.\\

% We will also prove the following complex version of Theorem \ref{thm:introthmconvexcone} which generalises the Kiselman minimum principle.  

The strategy of proof is as follows.  Suppose first that $f$ is $F\#\calPos$-subharmonic and smooth, and also that for each $x$ the minimum of $\{ f(x,y) : y\in \Omega_x\}$ is attained at some point $y=\gamma(x)$.  Assume also that $\gamma$ is sufficiently smooth.  Then a direct calculation, using the Chain Rule, allows one to express the $2$-jet of the marginal function $g$ in terms of the $2$-jet of the function $f$.  And from this it becomes clear that $f$ being $F\#\calPos$-subharmonic implies that $g$ is $F$-subharmonic.   The minimum principle for general $f$ is then reduced to this case through a number of approximations, the most significant being that since $F$ is assumed to be convex one can perform an integral mollification to eventually reduce to the case that $f$ is smooth.

One difference between the minimum principle we prove here and the classical case is that being convex (resp.\ plurisubharmonic) can be tested by restricting to lines (resp.\ complex lines), so it is essentially enough to assume that $X$ is one dimensional (i.e.\ that $n=1$).    For $F$-subharmonic functions we do not have this tool at our disposal.  However we will see that we can reduce to the case that the second variable lies in a one dimensional space (i.e.\ that $m=1$) which turns out to be crucial for the argument we give.\\

Notice that the convexity of $F$ is used only in the approximation part of the above argument (namely that the mollification of an $F$-subharmonic function remains $F$-subharmonic)s.  For this reason it seems to the authors that this convexity is a facet of the proof rather than an essential feature.  In a sequel to this paper we will take up the minimum principle again for non-convex subequations.\\

\noindent {\bf Comparison with other work: }The viscosity technique arose in the study of fully non-linear degenerate second-order differential equations, as pioneered by the work of Caffarelli--Nirenberg--Spruck \cite{CaffarelliNirenbergSpruckIII} and Lions--Crandall--Ishii (see the User's guide \cite{CILUserGuide} and references therein).     We use here the point of view taken by Harvey-Lawson \cite{HL_Dirichletduality,HL_Dirichletdualitymanifolds,HL_equivalenceviscosity} that replaces non-linear operators with certain subsets of the space of $2$-jets, which is in the same spirit as earlier work of Krylov \cite{Krylov}.

The minimum principle for convex functions is such a basic property that its origin appears to be time immemorial.   The minimum principle for plurisubharmonic functions was discovered by Kiselman \cite{KiselmanInvent,Kiselman}, and is often referred to as the Kiselman minimum principle.  This minimum principle is a tool used throughout pluripotential theory, and is known to be intimately connected to singularities of plurisubharmonic functions \cite{Demaillybook,Hormander_introductiontocomplexanalysis,KiselmanInvent} as well as the complex Homogeneous Monge-Amp\`ere Equation (see for instance the work of Darvas-Rubinstein \cite{DarvasRubinstein_rooftop},  or previous work of the authors \cite{RossNystrom_harmonicdiscs}).

 A stronger form of minimum principle for convex functions was proved by Prekopa \cite{Prekopa_concave}, which has since been extended to the plurisubharmonic setting by Berndtsson \cite{Berndtsson_Prekopa} (see also \cite{Erausquin}).   Other form of minimum principle can be found in \cite{Zeriahiminimum,Poletsky_minimumprinciple}, and surveys in \cite{DengZhangZhou_minimumprinciple,Kiselman_plurisubharmonicseveral}.
 
In \cite{RubinsteinDarvas}, Darvas-Rubinstein prove the first ``non-classical" minimum principle using the viscosity technique.  Their interest was in a particular subequation (that is constant coefficient and depends only on the Hessian part) related to the study of Lagrangian graphs.    After this work was complete we became aware that a more general minimum principle, with some similarities to the work developed here, appeared in a lecture course of Rubinstein.   We expect, but have not proved, that the minimum principle in \cite{RubinsteinDarvas} can be derived from the statement in this paper as the techniques used are rather similar (in fact both have close similarities with Kiselman's proof in \cite{Kiselman}).

The reader may also note that the authors have as an application of the main result in \cite{RN_argmin} a minimum principle that does not require any convexity of the subequation $F$, but instead requires a certain concavity assumption on the function $f(x,y)$. \\ 

%% Singularities of PSH functions; Demailly; 
%[K2] Kiselman: Densite des functions plurisubharmoniques
%tube domains

\noindent {\bf Possible Extensions and Future Directions:}.  The authors' interest in this minimum principle is an ongoing project that extends their work connecting the Hele-Shaw flow and the complex Homogeneous Monge-Amp\`ere Equation \cite{RossNystrom_harmonicdiscs}.  In fact it seems apparent that on can generalize the Hele-Shaw flow using F-subharmonic functions, and this flow is connected by a Legendre transformation to a Dirichlet problem involving the product subequation.  They key step in this is the minimum principle, and we plan to continue this in future work.

It is possible to extend the notion of $F$-subharmonicity to Riemannian manifolds  $X$ (see \cite{HL_Dirichletdualitymanifolds} and we observe that this requires some additional structure on $X$).     Being a local condition, it is clear that the work in this paper generalizes as well, and applies to products $X\times \mathbb R^m$.    We expect that one can also extend to the case that the $\mathbb R^m$ factor is replaced by a Riemannian manifold, but have not checked the details as this would go beyond our intended application.

We remark finally that the Harvey-Lawson theory also allows for quaternionic subequations.  It would be interesting to know if there is a minimum principle in this setting as well.\\

{\bf Organization:} In Section \ref{sec:productsofsubequations} we introduce the new notion of product subequation and make precise what we mean by a minimum principle for $F$-subharmonic functions.  Some examples are then given in Section \ref{sec:examples}.   In Section \ref{sec:approximations} we prove some general statements about approximating $F$-subharmonic functions by ones that are more regular (for instance the method of sup-convlution allowing approximation by semiconvex functions, and mollification that allows approximation by smooth ones).  

The main proofs appear in Sections \ref{sec:reductions} through Section \ref{sec:minimumcomplexconvex}, and proceed by a number ``reductions" that show that to prove our desired minimum principle for $f:X\times \mathbb R^m\to \mathbb R\cup \{-\infty\}$ we may, without loss of generality, make the following assumptions:
\begin{itemize}
\item we may assume $m=1$
\item $f$ may be assumed to be bounded from below (Proposition \ref{prop:mcanbeone})
\item $f$ may be assume to be relatively exhaustive. i.e. the minimum of $f$ on each fibre is obtained strictly away from the boundary (Proposition \ref{prop:reductionfibrewiseminimum})
\item $f$ may be assumed to be semiconvex (Proposition \ref{prop:reductionsemiconvex})
\item $f$ may be assumed to be smooth (Proposition \ref{prop:reductionsmooth})
\end{itemize}

After this we are left with the task of proving the minimum principle under all the above assumptions, which is done by direct calculation of the Hessian of $f$.  In the real case this can be found in Proposition \ref{prop:hessiancalc} and in the complex case Proposition \ref{thm:minimumprinciplecomplexconvexI}.\\

\noindent {\bf Acknowledgements: }The authors wish to thank Tristan Collins and Yanir Rubinstein for conversations that stimulated this work.

\section{Subequations}
In this section we recall the definition of subequations and $F$-subharmonic functions.   We start in the real case.  Let $M_{n\times m}(\mathbb R)=\Hom(\mathbb R^m,\mathbb R^n)$ denote the space of $n\times m$ real matrices,  $\Sym^2_{n}\subset M_{n\times n}(\mathbb R)$ denote the real symmetric $n\times n$ matrices and 
$$\Pos_{n} = \{ A\in \Sym^2_{n} :  v^tAv \ge 0 \text{ for all }v \}$$ denote the symmetric semipositive matrices. Assume $X\subset \mathbb R^{n}$ is open and let
$$J^2(X):=X\times \mathbb R\times \mathbb R^n \times \Sym^2_{n}$$
be the jet-bundle over $X$, which has fiber over $x\in X$
$$J^2_{n}: =  \mathbb R\times \mathbb R^n \times \Sym^2_{n}.$$
For $F\subset J^2(X)$ and $x\in X$ write 
$$ F_x= \{ (r,p,A)\in J^2_n  : (x,r,p,A)\in F\}.$$

\begin{definition}[Subequations]
We say that $F\subset J^2(X)$ is a \emph{primitive subequation} if 
\begin{enumerate}
\item (Closedness) $F$ is closed.
\item (Positivity) 
\begin{equation}\label{eq:positivity} 
(r,p,A)\in F_x \text{ and } P\in \Pos_{n} \Rightarrow (r,p,A+P)\in F_x.\end{equation}
\end{enumerate}

We say that $  F\subset J^2(X)$ is a \emph{subequation} if in addition
\begin{enumerate}\setcounter{enumi}{2}
\item (Negativity) \begin{equation}\label{eq:negativity} 
(r,p,A)\in F_x \text{ and } r'\le r \Rightarrow (r',p,A)\in F_x.\end{equation}
\item (Topological) \begin{equation}\label{eq:topological} F = \overline{\Int( F)} \text{ and } F_x = \overline{\Int F_x}.\text{ and } \Int F_x = (\Int F)_x\text{ for all }x\in X\end{equation}
\end{enumerate}
where the bar denotes topological closure.

We say that $ F$ is \emph{convex} if each $F_x$ is convex,  i.e. if $\alpha_1,\alpha_2\in  F_x$ and $t\in [0,1]$ then $t\alpha_1 + (1-t)\alpha_2\in F_x$.   %We say it is a \emph{cone} if each $F_x$ is cone over the origin,  i.e. if $\alpha\in  F_x$ then $t\alpha\in F_x$ for all $t\ge 0$.
\end{definition}

\begin{remark}
The majority of the results of this paper hold for primitive subequations that satisfy the Negativity Property.    The extra assumption of being a subequation is important for the Comparison Principle proved in \cite{HL_Dirichletduality,HL_Dirichletdualitymanifolds}.
\end{remark}

\begin{example}
Set
$$\calPos_{X}:= X\times \mathbb R\times \mathbb R^n \times \Pos_{n}$$
which is a convex subequation. By abuse of notation we will simply write $\calPos$ for $\calPos_X$ when $X$ is clear from context.% and write $\calPos_n$ for \calPos_{\mathbb R^n}$.
 \end{example}

\subsection{$F$-subharmonic functions}
We again let $X$ be an open subset of $\mathbb R^n$.

\begin{definition}[Upper contact points, Upper contact jets]
Let $$f: X\to \mathbb R\cup\{-\infty\}.$$ We say that $x\in X$ is an \emph{upper contact point} of $f$ if $f(x)\neq -\infty$ and there exists $(p,A)\in \mathbb R^n\times \Sym^2_n$ such that
$$f(y)\le f(x) + p.(y-x) + \frac{1}{2} (y-x)^tA(y-x) \text{ for all } y \text{ sufficiently near } x.$$
When this holds we refer to $(p,A)$ as an \emph{upper contact jet} of $f$ at $x$.
\end{definition}

\begin{definition}[F-subharmonic function]
Suppose $F\subset J^2(X)$.   We say that an upper-semicontinuous function $f:X\to \mathbb R\cup \{-\infty\}$ is \emph{F-subharmonic} if $$ (f(x),p,A)\in F_x \text{ for all upper contact jets }  (p,A) \text{ of }f \text{ at } x.$$
%and each upper contact jet $(p,A)$ of $f$ at $x$ it holds that $(f(x),p,A)\in F_x$.  
We let $F(X)$ denote the set of $F$-subharmonic functions on $X$.
\end{definition}

Observe that by definition any $x$ such that $f(x)=-\infty$ is not an upper-contact point, and so the function $f\equiv -\infty$ is trivially $F$-subharmonic. Clearly being $F$-subharmonic is a local condition, by which we mean that if $\{X_{\alpha}\}_{\alpha\in \mathcal A}$ is an open cover of $X$ then $f\in F(X)$ if and only if $f\in F(X_{\alpha})$ for all $\alpha$.  

\begin{example}[Convex and Plurisubharmonic]\label{example:convexity}
Recall $\mathcal P_X = X\times \mathbb R\times \mathbb R^n\times \Pos_n$.  Then $\mathcal P_X(X)$ consists of locally convex functions on $X$ \cite[Example 14.2]{HL_Dirichletdualitymanifolds}.  %Similarly if $X\subset \mathbb C^n$ is open then $\mathcal P^{\mathbb C}_X(X)$ consists of the plurisubharmonic functions on $X$ \cite[p63]{HL_Dirichletdualitymanifolds}.
\end{example}

For the reader's convenience, we collect the properties subequations and F-subharmonic functions from the work of Harvey-Lawson that we will need in Appendix \ref{appendix:subequations}.      The most pertinent is that if $f$ is $\mathcal C^2$, and $F$ satisfies the Positivity assumption, then $f$ is $F$-subharmonic if and only if all its second order jets lie in $F$ (see Lemma \ref{lem:fsubharmonicc2}).  For much more depth, including many interesting examples, the reader is referred to the original papers \cite{HL_Dirichletduality,HL_Dirichletdualitymanifolds,HL_equivalenceviscosity}.  In particular in \ref{sec:complexsubequations}
we discuss how this extends to the complex case (with the upshot being that for $\mathcal C^2$ functions the second order jet is to be understood in the complex sense).

\section{Products of Primitive Subequations}\label{sec:productsofsubequations}

\subsection{Real Products}
For $\Gamma\in \Hom(\mathbb R^{n},\mathbb R^{m})=M_{m\times n}(\mathbb R)$ consider
\begin{align}i_\Gamma:\mathbb R^{n} &\to \mathbb R^{n+m} \quad i_\Gamma(x) = (x,\Gamma x)\\
j : \mathbb R^{m} &\to \mathbb R^{n + m} \quad j(y) = (0,y).\end{align}
These induce natural pullback-maps
\begin{align*}
i_\Gamma^*: J^2_{n+m}\to J^2_{n}\\
j^* : J^2_{n+m}\to J^2_{m}\end{align*}
%\begin{align*}
%i_U^* : \mathbb R\times \mathbb R^{n+m} \times \Sym_{n+m}^2 &\to \mathbb R\times \mathbb R^{n} \times \Sym^2_{n}.\\
%j^* : \mathbb R\times \mathbb R^{n+m} \times \Sym_{n+m}^2 &\to \mathbb R\times \mathbb R^{m} \times \Sym^2_{m}.
%\end{align*}
with  the following property.  If $f:\mathbb R^{n+m}\to \mathbb R$ is twice differentiable at $(0,0)\in \mathbb R^{n+m}$ and we let \begin{align*}f_1(x) &:= f(x,\Gamma x) \text{ for } x\in \mathbb R^n\\f_2(y) &: = f(0,y)\text{ for } x\in \mathbb R^m\end{align*}
then $f_1,f_2$ are twice differentiable at $0\in \mathbb R^{n}$ and $0\in \mathbb R^{m}$ respectively, and
\begin{align}
 J^2_0 (f_1) &= i_\Gamma^* J^2_{(0,0)} (f)\\\
  J^2_{0} (f_2) &= j^* J^2_{(0,0)} (f).
 \end{align}
Explicitly suppose
$$ p := \left(\begin{array}{c} p_1 \\ p_2 \end{array} \right)\in \mathbb R^{n+ m}$$
and
$$ A := \left(\begin{array}{cc} B & C \\ C^t & D \end{array}\right)\in \Sym^2_{n+m}$$
where the latter is in block form, so $B\in \Sym^2_{n}$ and $D\in \Sym^2_{m}$.  Then
\begin{align}\label{eq:defofistar}
i_\Gamma^*(r,p,A) &= \left(r, p_1 + \Gamma^t p_2, B + C\Gamma + \Gamma^t C^t + \Gamma^tD\Gamma\right)\\
j^*(r,p,A)  &= (r,p_2,D).
\end{align}

\begin{comment}
\begin{definition}[Slices]
Given ${x_0}\in  X$ we call
$$ i_{x_0} : \{ y\in Y : (x_0,y)\in X\times Y \} \to X\times Y  \quad i_{x_0}(y) = (x,y)$$
a \emph{vertical slice} of $X\times Y$.  
Given $y_0\in Y$ and $\Gamma\in Hom(\mathbb R^{n},\mathbb R^{m})$ we call
$$ i_{y_0,\Gamma} : \{ x : (x,y_0 + \Gammax)\in X\times Y\} \to X\times Y \quad i_{y_0,\Gamma}(x) = (x,y_0 + \Gammax)$$
a \emph{non-vertical slice} of $X\times Y$.
\end{definition}

Vertical and non-vertical slices induce restriction maps
$$ i_{y_0,\Gamma}^* :  i_{y_0,U}^*J^2(X\times Y) \to J^2(X).$$
$$ i_{x_0}^*: i_{x_0}^*J^2(X\times Y) \to J^2(Y)$$
\end{comment}

\begin{definition}[Products]
Let $X\subset\mathbb R^n$ and $Y\subset \mathbb R^m$ be open, and $F\subset J^2(X)$ and $G\subset J^2(Y)$. Define $$F\#G \subset J^2(X\times Y)$$ by
%$$F\#G = \left \{  ((x,y),r,p,A)\in J^2(X\times Y) : \begin{array}{l}   i_{\Gamma}^*(r,p,A)) \in F_{x}  \text{ and }\\  j^*(r,p,A) \in G_y  \\ \text{for all }  \Gamma\in \Hom(\mathbb R^{n},\mathbb R^{m})\end{array}\right\}.$$
$$(F\#G)_{(x,y)}  = \left \{ \alpha\in J^2_{n+m} : \begin{array}{l}   i_{\Gamma}^*\alpha \in F_{x}  \text{ and }  j^*\alpha \in G_y  \\ \text{for all }  \Gamma\in \Hom(\mathbb R^{n},\mathbb R^{m})\end{array}\right\}.$$
\end{definition}

\begin{lemma}[Basic Properties of Products]\label{lem:productsofsubequations}\
\begin{enumerate}
\item If $F$ and $G$ both satisfy any of the following properties then the same is true for $F\#  G$:
\begin{inparaenum}
\item Positivity 
\item Negativity 
\item Being constant coefficient
\item Being independent of the gradient part
\item Convexity
%\item Being a cone 
\item Being closed 
\item Being a primitive subequation.
\end{inparaenum}
\item If $F_1\subset F_2\subset J^2(X) \text{ and } G_1\subset G_2\subset J^2(Y)$ then  $F_1\#G_1\subset F_2\#G_2.$
\item For $i=1,2$ let $H_i$ be a subgroup of $GL_{n_i}(\mathbb R)$ and suppose that $F_i\subset J^2(X_i)$ is $H_i$-invariant.  Then $F_1\#F_2$ is $H_1\times H_2$-invariant.
\end{enumerate}
\end{lemma}
\begin{proof}
Suppose $P\in \Pos_{n+m}$ and $\Gamma\in \Hom(\mathbb R^n,\mathbb R^m)$.   Then $i_\Gamma^*(0,0,P) = (0,0,P')$  and $j^*(0,0,P) = (0,0,P'')$ where $P'\in \Pos_{n}$ and $P''\in \Pos_{m}$.  Thus (a) follows by linearity of $i_\Gamma^*$ and $j^*$.  Statement (b,c,d) are immediate from the definition.  Statement (e)  follow from linearity of $i_\Gamma^*$ and $j^*$, and (f) from continuity of $i_\Gamma^*$ and $j^*$, and (g) combines (a) and (f).    Statements (2) is immediate, and (3) follows from identities such as
$$ i^*_{\Gamma} (h^*\alpha) = h_1^* (i_{h_2\Gamma h_1^{-1}}^*\alpha)  \text{ for } h= (h_1,h_2)\in H_1\times H_2$$
the details of which are left to the reader.
\end{proof}

\begin{remark}
Taking $H_1$ and $H_2$ to be the orthogonal group in Lemma \ref{lem:productsofsubequations}(3) allows us to extend the definition of $F_1\#F_2$ to products of Riemannian manifolds.  See \cite{HL_Dirichletdualitymanifolds}.
\end{remark}

\begin{proposition}[Associativity of products]\label{prop:associativityofproducts}
Let $X_i\subset \mathbb R^{n_i}$ be open and $F_i\subset J^2(X_i)$ for $i=1,2,3$.  Then
$$ (F_1\#F_2)\#F_3 = F_1\#(F_2\#F_3).$$
\end{proposition}
\begin{proof}
This can be proved by elementary manipulations, the details of which can found in Appendix \ref{sec:associativityproductsubequations}.
\end{proof}

\subsection{Products of Subequations}\label{subsec:productsofsubequations}

 We stress that the product of subequations is not necessarily a subequation since the Topological property may not be inherited (see Example \ref{sec:productsofsubequationnotasubequation}).   The following two statements give conditions under which this does hold.

\begin{definition}
We say that $F\subset J^2(X)$ has Property (P\textsuperscript{++}) if the following holds.  For all $x\in X$ and all $\epsilon>0$ there exists a $\delta>0$ such that
\begin{equation}(x,r,p,A)\in F_x \Rightarrow (x',r-\epsilon,p, A+\epsilon \Id)\in F_{x'} \text{ for all } \|x'-x\|<\delta. \tag{P\textsuperscript{++}}\label{eq:propertyH_original}
\end{equation}
\end{definition}

Clearly if $F$ a constant coefficient primitive subequation that has the Negativitiy Property then \eqref{eq:propertyH_original} holds.  %One may think of ({P\textsuperscript{++}}) as asking that the Positivity and Negativity properties properties hold locally uniformly in $x$.

\begin{lemma}[Products of gradient-independent subequations]\label{lem:productsofgradientindependent}
Assume  $F\subset J^2(X)$ and $G\subset J^2(Y)$ are subequations with property \eqref{eq:propertyH} that are independent of the gradient part.   Then $F\#G$ is a subequation.
\end{lemma}
\begin{proof}
The proof is elementary point-set topology, and can be found in Appendix \ref{appendixB}.
\end{proof}

\begin{corollary}[Products of constant-coefficient gradient-independent subequations]
If $F,G$ are both constant-coefficient primitive subequations with the Negativity Property that are independent of the gradient part, then $F\#G$ is a subequation.
\end{corollary}

\begin{remark}
Lemma \ref{lem:productsofgradientindependent} is far from optimal.  In fact we will give in Section \ref{sec:examples} examples of products of subequations that depend non-trivially on the gradient part that remain a subequation.
\end{remark}

\subsection{Complex Products}
Essentially the same definition is made in the complex case.   Given $\Gamma\in \Hom(\mathbb C^n,\mathbb C^m)$ we let $i_{\Gamma}(z) = (z,\Gamma z)$ and $j(w) = (0,w)$, which induce pullbacks
\begin{align*}
i_\Gamma^*: J^{2,\mathbb C}_{n+m}\to J^{2,\mathbb C}_{n}\\
j^* : J^{2,\mathbb C}_{n+m}\to J^{2,\mathbb C}_{m}\end{align*}
given explicitly by
\begin{align}\label{eq:defofistar_complex}
i_\Gamma^*(r,p,A) &= \left(r, p_1 + \Gamma^* p_2, B + C\Gamma + \Gamma^* C^* + \Gamma^*D\Gamma\right)\\
j^*(r,p,A)  &= (r,p_2,D),
\end{align}
where
$$ p := \left(\begin{array}{c} p_1 \\ p_2 \end{array} \right)\in \mathbb C^{n+ m}$$
and
$$ A := \left(\begin{array}{cc} B & C \\ C^* & D \end{array}\right)\in \Herm(\mathbb C^{n+m})$$
is in block form, so $B\in \Herm(\mathbb C^{n})$ and $D\in \Herm(\mathbb C^{m})$.

\begin{definition}[Products, Complex case]
Suppose $X\subset \mathbb R^{2n}\simeq \mathbb C^{n}$ and $Y\subset \mathbb R^{2m} \simeq \mathbb C^{m}$ are open.  Let $F\subset J^{2, \mathbb C}(X)$ and $G\subset J^{2, \mathbb C}(Y)$. Define
$$F\#_{\mathbb C}G \subset J^{2,\mathbb C}(X\times Y)$$ by
%$$F\#_{\mathbb C} G = \left \{  ((x,y),r,p,A)\in J^2(X\times Y) : \begin{array}{l}   i_{\Gamma}^*(r,p,A)) \in F_{x}  \text{ and }\\  j^*(r,p,A) \in G_y  \\ \text{for all }  \Gamma\in \Hom(\mathbb C^{n},\mathbb C^{m})\end{array}\right\}.$$
$$(F\#_{\mathbb C}G)_{(x,y)}  = \left \{ \alpha\in J^{2,\mathbb C}_{n+m} : \begin{array}{l}   i_{\Gamma}^*\alpha \in F_{x}  \text{ and }  j^*\alpha \in G_y  \\ \text{for all }  \Gamma\in \Hom(\mathbb C^{n},\mathbb C^{m})\end{array}\right\}.$$
\end{definition}

This definition is compatible with our abuse of notation in thinking of a complex $F\subset J^2(X)$ as being a subset of $J^{2,\mathbb C}(X)$.  That is, if $F'$ denotes the subset of $J^{2,\mathbb C}(X)$  we are associating with a complex primitive subequation $F\subset J^2(X)$, then $(F'\#_{\mathbb C} G') = (F\#G)'$.  For this reason all the basic properties we prove about real products extend immediately to the complex case (for example the analogue of Lemma \ref{lem:productsofsubequations} and Proposition \ref{prop:associativityofproducts} in the complex case follow immediately from the real case). \\

Just as plurisubharmonic functions remain plurisubharmonic after composition with a biholomorphism, the same is true for $F\#_{\mathbb C}\calPos^{\mathbb C}$-subharmonic functions after composition with biholomorphism in the second variable.   The precise statement is as follows:

\begin{lemma}[Composition with biholomorphisms in the second variable]\label{lem:compositionbiholomorphism}
Let $F\subset J^{2,\mathbb C}(X)$ be a complex primitive subequation, $\Omega\subset X\times \mathbb C^m$ be open and $f$ be  $F\#_{\mathbb C}\calPos^{\mathbb C}$-subharmonic on $\Omega$.    Suppose that $W\subset \mathbb C^m$ is open and $$\zeta:W\to \zeta(W)\subset \mathbb C^m$$
is a biholomorphism.  Then the function
$$\tilde{f}(z,w) : = f(z,\zeta(w))$$
is  $(F\#_{\mathbb C}\calPos^{\mathbb C})$-subharmonic on $\{ (z,w)\in X\times W: (z,\zeta(w))\in\Omega\}$.
\end{lemma}
\begin{proof}
Suppose first that $f$ is $\mathcal C^2$, so $J^{2,\mathbb C}_{(z,\zeta(w))} f$ exists.   Then $j^*J^{2,\mathbb C}_{(z,\zeta(w))} f\in \calPos^{\mathbb C}$ which implies $j^* J^{2,\mathbb C}_{(z,w)} \tilde{f} \in \calPos^{\mathbb C}$.  On the other hand if $\Gamma\in \Hom(\mathbb C^n,\mathbb C^m)$ then using the Chain Rule,
$$ i_{\Gamma}^* J^{2,\mathbb C}_{(z,w)} \tilde{f} = i_{\frac{\partial\zeta}{\partial w} \Gamma}^* J^{2,\mathbb C}_{(z,\zeta(w))} f$$
which lies in $F_z$ since $f$ is assumed to be  $F\#_{\mathbb C}\calPos^{\mathbb C}$-subharmonic.  This proves the result for $\mathcal C^2$-functions $f$. 

The case of upper-semicontinuous $f$ follows, since if $\tilde{\phi}$ is a $\mathcal{C}^2$ test-function touching $\tilde{f}$ from above at $(z_0,w_0)$ then  ${\phi}(z,w): = \tilde{\phi}(z,\zeta^{-1}(w))$ is a $\mathcal{C}^2$ test function touching $f$ from above at $(z_0,\zeta(w_0))$.    Thus the result follows from the first paragraph, combined with Lemma \ref{lem:viscositydefinition}.
\end{proof}

\subsection{$F\#G$-subharmonicity by restriction to slices}

We now prove that $f$ being $F\#G$-subharmonic is equivalent to the restriction of $f$ to each non-horizontal slice being $F$-subharmonic and its restriction to each horizonal slice being $G$-subharmonic.    Again $X\subset \mathbb R^{n}$ and $Y\subset \mathbb R^{m}$ are assumed to be open.

\begin{definition}[Slices]
Given ${x_0}\in  X$ we call
$$ j_{x_0} :\mathbb R^{m}\to \mathbb R^{n}\times \mathbb R^{m}  \quad j_{x_0}(y) = (x_0,y)$$
a \emph{vertical slice}.
Given $y_0\in Y$ and $\Gamma\in \Hom(\mathbb R^{n},\mathbb R^{m})$ we call
$$ i_{y_0,\Gamma} : \mathbb R^{n}  \to \mathbb R^{n}\times \mathbb R^{m}  \quad i_{y_0,\Gamma}(x) = (x,y_0 + \Gamma x)$$
a \emph{non-vertical slice}.
\end{definition}

\begin{proposition}\label{prop:slices}
Assume that $F$ and $G$ are constant coefficient primitive subequations.  Let $\Omega\subset X\times Y$ be open and $f:\Omega\to \mathbb R\cup \{-\infty\}$ be upper-semicontinuous.  The following are equivalent
\begin{enumerate}
\item $f$ is $F\#G$-subharmonic
\item 
$$i_{y_0,\Gamma}^*f \text{ is } F\text{-subharmonic for all } y_0\in Y\text{ and }\Gamma\in \Hom(\mathbb R^{n},\mathbb R^{m}) \text{ and} $$
$$j_{x_0}^*f \text{ is } G\text{-subharmonic for all } x_0\in X.$$
\end{enumerate}
The analogous statement holds in the complex case.
\end{proposition}
\begin{proof}
We prove the real case only as the complex one is the same.\\

(2)$\Rightarrow$ (1)  This follows from the definitions.  Let 
$$ \beta:=\left(\left(\begin{array}{c} p_1\\ p_2 \end{array}\right), \left(\begin{array}{cc} B & C \\ C^t & D \end{array}\right)\right)$$
be an upper contact jet for $f$ at a point $(x_0,y_0)$.  Recall this means
\begin{equation}\label{eq:sliceuppercontact} f(x,y) \le f(x_0,y_0) + \left(\begin{array}{c} p_1\\ p_2 \end{array}\right) . 
\left(\begin{array}{c} x-x_0\\ y-y_0 \end{array}\right)
+ \frac{1}{2} \left(\begin{array}{c} x-x_0\\ y-y_0 \end{array}\right)  ^t\left(\begin{array}{cc} B & C \\ C^t & D \end{array}\right)\left(\begin{array}{c} x-x_0\\ y-y_0 \end{array}\right)\end{equation}
for $(x,y)$ near $(x_0,y_0)$.  
We wish to show
$$ \alpha: = ((x_0,y_0), f(x_0,y_0),\beta) \in F\#G.$$

Fixing $x=x_0$ and letting $y$ vary in \eqref{eq:sliceuppercontact} show that that $(p_2,D)$ is an upper contact jet for $j^*_{x_0}f$ at the point $y_0$.   Since $j^*_{x_0}(f)\in G(Y)$ this implies
$$j^*\alpha = (y_0,f(x_0,y_0),p_2,D)\in (F_2)_{y_0}.$$  
Similarly putting $y=y_0+\Gamma (x-x_0)$ into \eqref{eq:sliceuppercontact} and letting $x$ vary gives that $(p_1+\Gamma ^t p_2, B + \Gamma C + C^t\Gamma ^t + \Gamma ^t D\Gamma )$ is an upper-contact jet for $i_{y_0,\Gamma }^*f$ at the point $x_0$.  Since $i_{y_0,\Gamma }^*f\in F(X)$ this gives
$$i_\Gamma ^*\alpha = (x_0,f(x_0,y_0),p_1 + \Gamma ^tp_2,B + C\Gamma  + \Gamma ^t C^t + \Gamma ^tD\Gamma )\in F_{x_0}.$$
As this holds for all $\Gamma $ we deduce that $\alpha\in F\#G$, and hence $f$ is $F\#G$-subharmonic.\\

(1)$\Rightarrow$ (2) This direction is more substantial, but follows easily from the Restriction  Theorem of Harvey-Lawson.   (The reader may want to observe that what we are calling a primitive subequation is called a subequation in  \cite{HL_Restriction} -- see the note after \cite[Definition 2.4]{HL_Restriction}).  Working locally there is no loss in assuming that $\Omega=X\times Y$.  If we let $i_{y_0,\Gamma }^*$ and $j^*_{x_0}$ also denote the pullback on second-order jets then, essentially by definition,
$$ i_{y_0,\Gamma }^* (F\#G) = F \text{ and } j_{x_0}^*(F\#G) = G$$
which are both closed.    Thus as $F\#G$ is assumed constant coefficient, \cite[Theorem 5.1]{HL_Restriction} applies to give (2).
\end{proof}

\subsection{Boundary convexity and definition of the minimum principle}

Recall that a function $u:X\to \mathbb R$ defined on an open $X\subset \mathbb R^{n}$ is \emph{exhaustive} if for each $a\in \mathbb R$ the sublevel set
$$\{ x\in X : u(x)<a\}$$
is relatively compact in $X$.

\begin{definition}[$F$-pseudoconvexity]\label{def:Fpseudoconvexity}
Let $F\subset J^2(X)$.  We say that $\Omega\subset X$ is $F$-\emph{pseudoconvex} if there exists a
$$u:\Omega\to \mathbb R$$
that is continuous, exhaustive and $F$-subharmonic. 
\end{definition}

\begin{remark}
In \cite[Section 5]{HL_Dirichletduality} there is a notion of ``boundary convexity" defined for domains with smooth boundary, which is possibly different to being $F$-pseudoconvex.
\end{remark}

Let $\pi:\mathbb R^{n+m}\to \mathbb R^m$ be the projection.  
For a subset $\Omega\subset X\times \mathbb R^m$ and $x\in X$ we write
$$\Omega_x: = \pi^{-1}(x)  = \{ y\in \mathbb R^m : (x,y) \in \Omega\}.$$

\begin{definition}[Connected fibres]
We say that $\Omega$ has \emph{connected fibres} if $\Omega_x$ is connected for each $x\in \pi(\Omega)$.
\end{definition}

\begin{lemma}\label{lem:slicesareconvex}
Assume that $F\subset J^2(X)$  is a constant coefficient primitive subequation, that and $f$ is $F\#\calPos$-subharmonic on an open $\Omega\subset X\times \mathbb R^{m}$.   Then   
\begin{enumerate} 
\item The function $h:\Omega_x\to \mathbb R\cup \{-\infty\}$ given by $h(y) = f(x,y)$ is locally convex.  
\item If $\Omega_x$ is connected and $h$ is exhaustive then $\Omega_x$ is convex.
%\item If $\Omega_x$ is connected and $h$ is exhausting then for any $a\in \mathbb R$ the set 
%$$\{ y\in \Omega_x : f(x,y)<a \}$$
%is convex (so in particular connected).
\end{enumerate}
\end{lemma}
\begin{proof}
Proposition \ref{prop:slices} says that $h$ is $\mathcal P$-sub\-harmonic, which means it is locally convex (Example \ref{example:convexity}) giving (1).  Statement (2) follows as a connected open set admitting an exhaustive locally convex function is convex \cite[Theorem 2.1.25]{Hormander_notionsofconvexity}.   %For statement (3) let $y_0,y_1\in S:= \{ y\in \Omega_x : f(x,y)<a \}$, so $h(y_0)<a$ and $h(y_1)<a$.    By (2) the line segment joining $y_0$ and $y_1$ lies in $\Omega_x$, and so by convexity of $h$ we have $h<a$ everywhere on this line segment.  Thus the entire line segment is contained in $S$ proving it is both connected and convex.
\end{proof}

\begin{corollary}\label{cor:slicesofpseudoconvexareconvex}
If $\Omega\subset X\times \mathbb R^m$ is a $F\#\calPos$-pseudoconvex domain and $\Omega_x$ is connected then $\Omega_x$ is convex.
\end{corollary}
\begin{proof}Let $f$ be $F\#\calPos$-subharmonic and exhaustive on $X$.  Then for each $x\in X$ the function $h(y)= f(x,y)$ is exhausing on $\Omega_x$ so Lemma \ref{lem:slicesareconvex}(2) applies.
\end{proof}

\begin{definition}[The minimum principle in the real case]\label{def:minimumprinciple}
Let $F\subset J^2(X)$ be a primitive subequation.  We say that $F$ satisfies the \emph{minimum principle} if the following holds:
 
Let $\Omega\subset X\times \mathbb R^{m}$ be a $(F\#\calPos)$-pseudoconvex domain with connected fibres.  Then $\pi(\Omega)$ is $F$-pseudoconvex, and for any $f$ that is $F\#\calPos$-pseudoconvex on $\Omega$, the marginal function 
 $$g(x) : = \inf_{y\in \Omega_x} f(x,y)$$
is $F$-subharmonic on $\pi(\Omega)$.
\end{definition}

In the complex case we make essentially the same definition, only requiring also that $\Omega$ and $f$ be independent of the argument of the second variable.  Consider the real torus
$$\mathbb T^m : = \{e^{i\theta}:=  (e^{i\theta_1},\cdots,e^{i\theta_m}): \theta = (\theta_1,\ldots, \theta_m)\}.$$
which acts on $\mathbb C^m$  by
$$e^{i\theta} w =e^{i\theta} \left(\begin{array}{c} w_1 \\ \vdots \\ w_m\end{array}\right)  =  \left(\begin{array}{c} e^{i\theta_1} w_1 \\ \vdots \\ e^{i\theta_m} w_m\end{array}\right).$$
\begin{definition}[$\mathbb T^m$-invariance]
We say that $\Omega\subset \mathbb C^n\times \mathbb C^m$ is \emph{$\mathbb T^m$-invariant} if
$$(z,w)\in \Omega   \Rightarrow (z,e^{i\theta} w)\in \Omega \text{ for all } e^{i\theta}\in \mathbb T^m.$$
When this holds we say a function $f:\Omega\to \mathbb R\cup \{-\infty\}$ is \emph{$\mathbb T^m$-invariant} if
$$ f(z,w) = f(z,e^{i\theta} w) \text{ for all } e^{i\theta}\in \mathbb T^m.$$
\end{definition}

\begin{definition}[Minimum Principle in the Complex Case]\label{def:minimumprinciple_complex}
Let $X\subset \mathbb C^n$ be open.  We say a complex primitive subequation $F\subset J^{2,\mathbb C}(X)$ satisfies the \emph{minimum principle} if the following holds:

Suppose $\Omega\subset X\times \mathbb C^{m}$ is $\mathbb T^m$-invariant $(F\#_{\mathbb C}\calPos^{\mathbb C})$-psuedoconvex domain with connected fibres.    Then  $\pi(\Omega)$ is $F$-pseudoconvex, and if $f$ is a $\mathbb T^m$-invariant $F\#_{\mathbb C}\calPos^{\mathbb C}$-subharmonic function on $\Omega$  then the marginal function
$$g(z) : = \inf_{w\in \Omega_z} f(z,w)$$ is $F$-subharmonic on $\pi(\Omega)$. 
\end{definition} 

In both the real and complex case we have expressed this minimum principle for the product $F\#\calPos$.  However it is easy to see that one can replace the factor $\calPos$ with any smaller primitive subequation, as in the following statement.

\begin{lemma}
Let $X\subset \mathbb R^{n}$ and $Y\subset \mathbb R^{m}$ be open.  Assume that $F\subset J^2(X)$ and $G\subset J^2(Y)$ are primitive subequations such that
\begin{enumerate}
\item $F$ satisfies the minimum principle.
\item $G\subset \calPos$
\end{enumerate}
Then for any $F\#G$-pseudoconvex domain $\Omega \subset X\times Y$ and any $F\#G$-subharmonic function $f$ on $\Omega$ the marginal function
$$g(x) = \inf_{y\in \Omega_x} f(x,y)$$
is $F$-subharmonic on the $F$-pseudoconvex domain $\pi(\Omega)$.  The analogous statement holds in the complex case.
\end{lemma}
\begin{proof}
We consider $\Omega \subset X\times Y \subset X\times \mathbb R^{m}$.  Since $G\subset \calPos$, if $\Omega$ is $F\#G$-psudoconvex then it is also $F\#\calPos$-pseudoconvex, and if $f$ is $F\#G$-subharmonic on $\Omega$ then it is also $F\#\calPos$-subharmonic.  Hence the statement that we want follows from the minimum principle for $F$.
\end{proof}

\begin{remark}[Functions independent of the imaginary part of the second variable]
The Kiselman minimum principle is often stated for pseudoconvex domains $\Omega\subset \mathbb C^n\times \mathbb C^m$ that are independent of the \emph{imaginary} part of the second variable (rather than, as we have done here, independent of the argument of the second variable).   But these are equivalent by considering the map $\phi(z,w) := (z,e^{w})$ and observing that $\Omega$ is pseudoconvex and independent of the imaginary part of the second variable if and only if $\phi(\Omega)$ is pseudoconvex and independent of the argument of the second variable.

When dealing with a general complex subequation $F$ it is not clear if $\Omega$ being $F\#_{\mathbb C}\calPos$-pseudoconvex implies that $\phi(\Omega)$ is also $F\#_{\mathbb C}\calPos^{\mathbb C}$-pseudoconvex (this would follow if $\Omega$ admitted an exhaustive $F\#_{\mathbb C}\calPos^{\mathbb C}$-subharmonic function that was independent of the imaginary part of the second variable, but the authors do not know whether this always holds).   However by an averaging argument we will later prove that a $\mathbb T$-invariant  $F\#_{\mathbb C}\calPos^{\mathbb C}$-pseudoconvex domain always admits a $\mathbb T$-invariant exhaustive $F\#_{\mathbb C}\calPos^{\mathbb C}$-subharmonic function (see Lemma \ref{lem:exhaustionindependent}).  For this reason we have stated our minimum principle in terms of the argument of the second variable, rather than its imaginary part (but see also Corollary \ref{cor:minimumimaginary}).
\end{remark}

\section{Examples}\label{sec:examples}

%\subsection{Convexity and Plurisubharmonicity}\label{sec:exampleconvexity}
\subsection{Example I}\label{sec:exampleI}
It is not hard to check directly from the definitions that
\begin{align}
\calPos_{\mathbb R^{n}}\# \calPos_{\mathbb R^{m}}&= \calPos_{\mathbb R^{n+m}}\text{ and }\label{eq:productrealpositive}\\
\calPos^{\mathbb C}_{\mathbb C^{n}}\# \calPos^{\mathbb C}_{\mathbb C^{m}}&= \calPos^{\mathbb C}_{\mathbb C^{n+m}}. \label{eq:productcomplexpositive}
\end{align}
In the following we will give another example with a similar property, that need not be constant coefficient or independent of the gradient part.

\subsection{Example II}

Suppose $X\subset \mathbb R^n$ is open and let \begin{align*}
\lambda&:\mathbb R\to \mathbb R\\
\mu&:X\times \mathbb R\to \Sym^2_n(\mathbb R)\end{align*}
\begin{definition}
Define 
$$F_{\lambda,\mu} = \{ (x,r,p,A)\in  J^2(X) : A \ge  \lambda p p^t + \mu \text{ where } \lambda = \lambda(r) \text{ and } \mu= \mu(x,r)\}.$$
\end{definition}

\begin{lemma}\
\begin{enumerate}
\item $F_{\lambda,\mu}$ has the Positivity Property.
\item If $\lambda$ and $\mu$ are continuous then $F$ is closed and has the Topological Property.
\item If $\mu(x,r)$ and $\lambda(r)$ are monotonic increasing in $r$ then $F_{\lambda,\mu}$ has the Negativity Property.
\item If $\lambda(r)=\lambda_0\ge 0$ for all $r$, and $\mu(x,r)$ is convex in $r$ then $F_{\lambda,\mu}$ is convex.
\item If $\lambda(r)\ge 0$ and $\mu(x,r)\ge 0$ for all $x,r$ then $F_{\lambda,\mu}\subset\calPos$.
\end{enumerate}
In particular if both conditions (2) and (3) hold then $F_{\lambda,\mu}$ is a subequation.

\end{lemma}
\begin{proof}
(1,3,5) are immediate from the definition, and (2) is simple point-set topology that is left to the reader.  (4) follows as the function $p\mapsto pp^t$ is convex.
\end{proof}

\begin{proposition}\label{prop:characterizationexample}
Suppose that for $i=1,2$ we have open $X_i\subset \mathbb R^{n_i}$.  Let
\begin{align*}
 \lambda &: \mathbb R\to \mathbb R\\
\mu_1 &: X_1\times \mathbb R\to \Sym^2_n.
\end{align*}
Define
$\mu:  X_1\times X_2\times \mathbb R\to \Sym^2_{n_1+n_2}$ by
$$\mu(x_1,x_2,r) := \diag(\mu_1(x_1,r), 0)$$
(here $\diag$ is to be understood as the matrix in block diagonal form).
Then
$$F_{\lambda,\mu_1} \# F_{\lambda,0} = F_{\lambda,\mu}.$$
\end{proposition}

In particular, the previous proposition gives various examples of products that are subequations even though they are not gradient-independent (compare Section \ref{sec:productsofsubequations}).

\begin{proof}[Proof of Proposition \ref{prop:characterizationexample}]
Write $F_1: = F_{\lambda,\mu_1}$ and $F_2: = F_{\lambda,0}$.    Fix $x_i\in X_i$ and let
$$\alpha:=\left(r, p=\left( \begin{array}{c}p_1 \\ p_2\end{array}\right) ,\left( \begin{array}{cc}B & C \\ C^t & D\end{array}\right)\right) \in J^2_{n+m}.$$
To ease notation set $\lambda: = \lambda(r)$  $\mu_1:=\mu_1(x_1,r)$ and
$$\left( \begin{array}{cc}\hat{B} & \hat{C} \\ \hat{C}^t & \hat{D}\end{array}\right) : = \left( \begin{array}{cc}B & C \\ C^t & D\end{array}\right)-\lambda pp^t - \mu(x_1,x_2,r)$$
so
\begin{align*}
\hat{B} &= B - \lambda p_1p_1^t - \mu_1\\
 \hat{C} &= C - \lambda p_1p_2^t\\
\hat{D}&=D - \lambda p_2p_2^t.
\end{align*}

{\bf Claim: } $\alpha\in (F_1\#F_2)_{(x_1,x_2)}$ if and only if the following three conditions all hold:
 %\begin{align}
  % \left(\Id -    (D-\lambda  p_1p_1^t)  (D-\lambda p_1p_1^t)^{\dagger}\right)(C-\lambda p_1p_2^t) &= 0 \label{ex1:condition1}\\
 %    (B-\lambda p_1p_1^t)  - (C-\lambda p_1p_2^t)^t (D-\lambda p_1p_1^t)^{\dagger} (C-\lambda p_1p_2^t) &\ge 0 \label{ex1:condition2}, \\
 %    D-\lambda p_2p_2^t&\ge 0 \label{ex1:condition3},
% \end{align}
  \begin{align}
   \left(\Id -    \hat{D}  \hat{D}^{\dagger}\right)\hat{C}^t &= 0 \label{ex1:condition1}\\
     \hat{B}  - \hat{C} \hat{D}^{\dagger} \hat{C}^t &\ge 0 \label{ex1:condition2}, \\
     \hat{D}&\ge 0 \label{ex1:condition3},
 \end{align}
 (where $\hat{D}^{\dagger}$ denotes the pseudo-inverse of $\hat{D}$).

The equivalence between these three conditions and the condition that
\begin{equation}\left(\begin{array}{cc} B  & C  \\ C^t & D \end{array}\right)\ge \lambda pp^t + \mu \label{lem:characterizationexamplepos:repeat}\end{equation}
is a standard piece of Linear Algebra (Proposition \ref{prop:positiveblock}), and this is precisely the condition that $\alpha\in (F_{\lambda,\mu})_{(x_1,x_2)}$.\\

To prove the claim, recall that by definition $\alpha\in (F_1\#F_2)_{(x_1,x_2)}$ if and only if 
\begin{enumerate}[(i)]
\item $j^*\alpha = (r,p_2,D)\in (F_2)|_{x_2}$ 
\item $\iota_{\Gamma}(\alpha)\in (F_1)_{x_1}$ for all $\Gamma$.
\end{enumerate}
Now by definition of $F_2$, (i) is equivalent is equivalent to $D\ge \lambda p_2p_2^t$ which in turn is equivalent to $\hat{D}\ge 0$.  To better understand (ii) write
\begin{align*}\zeta(\Gamma)&:=B + C\Gamma + \Gamma^tC^t + \Gamma^t D\Gamma- \lambda (p_1+\Gamma^tp_2) (p_1+\Gamma^t p_2))^t -\mu_1\\
&=\hat{B} +  \hat{C} \Gamma  + \Gamma^t \hat{C}^t + \Gamma^t \hat{D} \Gamma.
\end{align*}

%$$\zeta(\Gamma):=B + C\Gamma + \Gamma^tC^t + \Gamma^t D\Gamma-  (\lambda_1 p_1+\Gamma^t \lambda_2 p_2)(\lambda_1 p_1+\Gamma^t \lambda_2 p_2)^t$$
Then 
$$\iota_\Gamma^*(\alpha)\in (F_1)_{x_1} \Leftrightarrow \zeta(\Gamma)\ge 0.$$

Suppose first that \eqref{ex1:condition1},\eqref{ex1:condition2},\eqref{ex1:condition3} all hold.    The first and third of these state that $(I-\hat{D}\hat{D}^\dagger) \hat{C}^t =0$ and $\hat{D}\ge 0$.   So by Lemma \ref{lem:minquadraticmatrix}
$$\zeta(\Gamma) \ge \hat{B} - \hat{C}^t \hat{D}^\dagger \hat{C}\ge 0.$$
As this holds for all $\Gamma$ we deduce $\alpha\in (F_1\#F_2)_{(x_1,x_2)}$.

In the other direction, suppose that one of \eqref{ex1:condition1},\eqref{ex1:condition2},\eqref{ex1:condition3} do not hold. If it if \eqref{ex1:condition3} that does not hold then $j^*\alpha \notin (F_2)_{x_2}$ which immediately implies that $\alpha\notin (F_1\#F_2)_{(x_1,x_2)}$.  So we may suppose that $\hat{D}\ge 0$. 

If \eqref{ex1:condition1} does not hold there is some $w\neq 0$ such that $(I- \hat{D}\hat{D}^{\dagger})\hat{C}^tw\neq 0$.  Set $b:=\hat{C}^tw$ and let
$$q_{\hat{D},b}:=2 x^tb+ x^t \hat{D} x$$
Then by Lemma \ref{lem:minquadratic} there is a sequence $x_j$ such that $q_{\hat{D},b}(x_j)\to -\infty$ as $j\to \infty$.  For each $j$ pick a $\Gamma_j\in M_{n_1\times n_2}(\mathbb R)$ so that $x_j = \Gamma_j w$.   Then
$$ w^t(\hat{C} \Gamma_j +\Gamma_j^t\hat{C}^t + \Gamma_j^t \hat{D} \Gamma_j) w = b^t x_j + x_j^tb + x_j^t\hat{D} x_j  = q_{\hat{D},b}(x_j)\to -\infty$$
as $j\to \infty$.  But this implies $\zeta(\Gamma_j)$ is not semipositive for $j$ sufficiently large,  and so $\alpha\notin (F_1\#F_2)_{(x_1,x_2)}$.

Finally assume \eqref{ex1:condition1} and \eqref{ex1:condition3} both hold, but \eqref{ex1:condition2} does not. Then there is a $w\neq 0$ such that 
$$ w^t(\hat{B} - \hat{C} \hat{D}^{\dagger} \hat{C}^t) w <0.$$
Again set $b:= \hat{C}^tw$.   Then Lemma \ref{lem:minquadratic} implies is an $x_0\in \mathbb R^n$ such that
$$ q_{\hat{D},b}(x_0)  = - b^t \hat{D}^{\dagger} b$$
Now choose $\Gamma\in M_{n\times m}(\mathbb R)$ so $x_0 = \Gamma w$.   Then
$$w^t\zeta(\Gamma)w = w^t\hat{B}w + q_{\hat{D},b}(x_0) = w^t(\hat{B} - \hat{C}\hat{D}^{\dagger}\hat{C}^t)w<0$$
Hence $\zeta(\Gamma)$ is not semipositive, and so $\alpha\notin (F_1\#F_2)_{(x_1,x_2)}$.  
\end{proof}

\begin{lemma}\label{lem:minquadratic}
Let $P\in \Sym^2(\mathbb R^n)$ and $b\in \mathbb R^n$.   Then the function
\begin{equation}\label{def:q}
q_{P,b}(x) := 2 x^tb+ x^t P x\end{equation}
is bounded from below if and only if $P\ge 0$ and $(I - P P ^{\dagger})b=0$, in which case the minimum value is attained and is equal to
$$p^*:=-b^tP^{\dagger} b.$$
\end{lemma}
\begin{proof}
See \cite[p 420]{Gallier_Geometricmethods} or \cite[Proposition 4.2]{Gallier}.
\end{proof}

\begin{lemma}\label{lem:minquadraticmatrix}
Let $D\in \Pos_{m}$ and $C\in M_{n\times m}$ with $(I - D D ^{\dagger})C^t=0$.  For each $\Gamma\in M_{m\times n}$ let
$$Q(\Gamma)  : = C\Gamma + \Gamma^tC^t + \Gamma^t D \Gamma.$$
Then
\begin{equation}\label{lowerboundQ}Q(\Gamma) \ge - C^t D^{\dagger} C \text{ for all } \Gamma.\end{equation}
\end{lemma}
\begin{proof}
Fix a vector $v$ and consider
$$ v^t Q(\Gamma) v = v^t C \Gamma v + v^t \Gamma^t C^t v + v^t \Gamma^t D \Gamma v.$$
Defining $x:=\Gamma v$ and $b:=C^t v$ this becomes
$$ v^t Q(\Gamma) v = b^t x + x^t b + x^t D x = q_{D,b}(x)$$
where $q_{D,b}$ is as in \eqref{def:q}.   Note that $(I-DD^{\dagger}) b = (I-DD^{\dagger})C^t v =0$.  Hence Lemma \ref{lem:minquadratic} applies giving
$$ v^t Q(\Gamma) v \ge - b^t D^{\dagger} b = -v^t C D^{\dagger} C^t v.$$
As this holds for all $v$ we conclude \eqref{lowerboundQ}.
\end{proof}

\begin{proposition}\label{prop:positiveblock}
Consider a symmetric block matrix
$$ M := \left(\begin{array}{cc} \hat{B} & \hat{C} \\ \hat{C}^t & \hat{D}\end{array}\right).$$ 
Then $M\ge 0$ if and only if (i) $\hat{D}\ge 0$  (ii) $(I- \hat{D} \hat{D}^{\dagger}) \hat{C}^t=0$ and  (iii) $\hat{B} - \hat{C} \hat{D}^{\dagger} \hat{C}^t\ge 0$.
\end{proposition}
\begin{proof}
See \cite[Theorem 16.1]{Gallier_Geometricmethods} or \cite[Theorem 4.3]{Gallier}.
\end{proof}

\begin{remark}\
\begin{enumerate}
\item An analogous statement holds in the complex case; details are left to the reader.
\item Putting $\lambda\equiv 0$ and $\mu\equiv 0$ recovers \eqref{eq:productrealpositive} and \eqref{eq:productcomplexpositive}.
\end{enumerate}
\end{remark}

\subsection{Example 3}\label{sec:productsofsubequationnotasubequation}

Let $F\subset J^2(\mathbb R^n)$ be given by
$$F_x = \{ (r,p,A) : \|p\|\le 1\}$$
which is easily checked to be a constant-coefficient subequation.   If $(r,\left(\begin{array}{c} p_1 \\ p_2\end{array}\right),A')\in (F\#\calPos)_{(x,y)}$ then $\|p_1 + \Gamma^t p_2 \|\le 1$ for all $\Gamma$, which happens if and only if $p_2=0$ and $\|p_1\|\le 1$.    Thus $F\#\calPos$ is not the closure of its interior (so does not satisfy the Topological Property) and thus is not a subequation.

\section{Approximations}\label{sec:approximations}

In this section we describe three approximation techniques.  The first gives an approximation of $F\#\calPos$-subharmonic functions by $F\#\calPos$-sub\-harmonic functions that are bounded from below.  The second describes the sup-convolution that approximates bounded $F$-subharmonic functions by semiconvex ones.   The third uses smooth mollification to approximate continuous $F$-subharmonic functions by smooth ones.  We remark that it is only the third of these that requires $F$ to be convex.

\subsection{Approximation by subharmonic functions bounded from below}

If $F$ is such that sufficiently negative constant functions are $F$-subharmonic  then any $F\#\calPos$-subharmonic function $f$ can be approximated by the $F\#\calPos$-sub\-harmonic functions
$$ f_j = \max\{ f, -j\}$$
which are bounded from below and decrease to $f$ pointwise as $j$ tends to infinity.

In the following three statements we will show that essentially the same can be arranged if $F$ is constant-coefficient, without assuming that sufficiently negative constant functions are $F$-subharmonic.

\begin{lemma}\label{lem:boundedbelow}
Let $F\subset J^2(X)$ be a primitive constant-coefficient subequation.  Assume there exists an $F$-subharmonic function on $X$ that is not identically $-\infty$.  Then any $x_0\in X$ is contained in a neighbourhood $x_0\in U\subset X$ on which there exists an $f\in F(U)$ that is bounded from below on $U$.
\end{lemma}
\begin{proof}
Let $u\in F(X)$ be not identically $-\infty$ and choose some $x_0\in X$ such that $a:=u(x_0)\neq -\infty$.  As $u$ is upper-semicontinuous the set $U_1:=\{ x\in X: f(x)<a+1\}$ is an open neighbourhood of $x_0$.  Choose some $\delta>0$ so that $B_{\delta}(x_0)\subset U_1$ and let $U: =B_{\delta/2}(x_0)$.

For $\|\zeta\|<\delta/2$ define
$$ u_{\zeta}(x) : = u(x-\zeta) \text{ for } x\in U.$$
Observe that if $x\in U$ then $\|x-\zeta-x_0\|\le \|x-x_0\| + \|\zeta\|<\delta$, so $u_{\zeta}$ is well defined and bounded above by $a+1$.    Moreover, as $F$ is constant-coefficient, $u_{\zeta}$ is $F$-subharmonic on $U$.    Thus
$$ v : = {\sup}^*\{ u_{\zeta} : \|\zeta\|<\delta/2 \}$$
is $F$-subharmonic on $U$.  And if $x\in U$ then setting $\zeta':= x-x_0$ we have $\|\zeta'\|<\delta/2$ and so $v(x) \ge u_{\zeta'}(x) = u(x-\zeta') = u(x_0)=a$ and so $v$ is bounded from below by $a$ on $U$.

Thus provides an $F$-subharmonic function bounded from below on a neighbourhood $U$ of the point $x_0$.  But as $F$ is constant coefficient, we can translate $U$ to cover any given point of $X$, completing the proof.
\end{proof}

\begin{lemma}\label{lem:pullbacksubharmonic}
Let $F\subset J^2(X)$ be a constant-coefficient primitive subequation.  Then for any $F$-subharmonic function $f$ on $X$,  the function $$\pi^*f(x,y) := f(x) \text{ for } (x,y)\in X\times \mathbb R^{m}$$
is $F\#\calPos$-subharmonic.
\end{lemma}
\begin{proof}
The vertical slices of $\pi^*f$ are constant (so certainly $\mathcal P$-subharmonic) and the restriction of $\pi^*f$ to any non-vertical slice is equal to $f$ and hence $F$-subharmonic.  Thus $\pi^*f$ is $F\#\calPos$-subharmonic from Proposition \ref{prop:slices}.

%For an alternative proof, observe that if the second order jet of $\pi^* f$ exists then it is clearly in $F\#\calPos$.  Thus the result holds by the Almost Everywhere Theorem (Theorem \ref{thm:ae}) for semiconvex functions $f$, and using a sup-convolution (as discussed in the next section) one gets the result in general.
\end{proof}

\begin{proposition}\label{prop:approxboundedbelow}
Let $F\subset J^2(X)$  be a constant coefficient primitive subequation that has the Negativity Property.   Let $\Omega\subset X\times \mathbb R^{m}$ be open and $f$ be $F\#\calPos$-subharmonic on $\Omega$ that is not identically $-\infty$.

Then given any $x_0\in \pi(\Omega)$ there exists a neighbourhood $x_0\subset U \subset \pi(\Omega)$ and a $v\in F(U)$ that is bounded from below.  Moreover setting
$$ \Omega_U := \{ (x,y)\in \Omega : x\in U\}$$
the functions
$$ f_j = \max\{ f, \pi^*v -j\} \text{ on } \Omega_U$$
are $F\#\calPos$-subharmonic, are bounded below, and decrease pointwise to $f$ on $\Omega_U$ as $j\to\infty$.
\end{proposition}
\begin{proof}
If the only $F$-subharmonic functions on $X$ are identically $-\infty$ then $f$ must also be identically $-\infty$ so there is nothing to prove.  Otherwise given $x_0\in \pi(\Omega)$ Lemma \ref{lem:boundedbelow} provides a neighbourhood $x_0\subset U\subset \pi(\Omega)$ and an $F$-subharmonic function $v$ on $U$ that is bounded from below.  Furthermore Lemma \ref{lem:pullbacksubharmonic} says that $\pi^*v$ is $F\#\calPos$-subharmonic on $U\times \mathbb R^{m}$, and so in particular is $F\#\calPos$-subharmonic on $\Omega_U$.  Thus the $f_j$ are $F\#\calPos$-subharmonic, bounded from below, and decrease pointwise to $f$ on $\Omega_U$.
\end{proof}

\subsection{Sup-convolutions}\label{sec:supconvolution}
 
The Sup-convolutions is a well-known construction in the theory of viscosity subsolutions,  that allows approximation of certain $F$-subharmonic functions by semiconvex functions.     % In the following let $X\subset \mathbb R^{n}$ be open, $F\subset J^2(X)$ be a primitive subequation.  %Assume throughout this section that
%$$ F\text{ has constant coefficients}.$$

\begin{definition}[Semiconvexity]
Let $\kappa\ge 0$ and $X\subset \mathbb R^n$ be open.  We say $f:X\to \mathbb R$ is \emph{$\kappa$-semiconvex} if $f(x) + \frac{\kappa}{2} \|x\|^2$ is convex.  If $f$ is $\kappa$-semiconvex for some $\kappa \ge 0$ then we say simply $f$ is \emph{semiconvex}.
\end{definition}

\begin{lemma}\label{lem:supconvolutionsemiconvex}
Let $X\subset \mathbb R^n$ be open and $f:X\to \mathbb R$.    Given $\epsilon>0$ and non-empty $X'\subset X$ the function
$$g(x):= \sup_{\zeta\in X'}\{ f(\zeta) - \frac{1}{2\epsilon} \|x-\zeta\|^2\} \text{ for } x\in X.$$
is $\frac{1}{2\epsilon}$-semiconvex on $X$.
\end{lemma}
\begin{proof}
Write
$$ g(x) + \frac{1}{2\epsilon} \|x\|^2 = \sup_{\zeta\in X'} \{ f(\zeta) - \frac{1}{2\epsilon} \|\zeta\|^2 + \frac{1}{\epsilon} x\cdot \zeta\}$$
which is a supremum of affine functions in $x$, and thus is convex.
\end{proof}

\begin{lemma}[Sup-convolution]\label{lem:supconvolution}
Let $X\subset \mathbb R^n$ be open.   Let $F\subset J^2(X)$ be a constant coefficient primitive subequation that has the Negativity property.  Suppose $f$ is $F$-subharmonic on $X$ and bounded, say $|f(x)|<M$ for $x\in X$.  

Define the \emph{sup-convolution}
$$f^{\epsilon}(x) = \sup_{\zeta\in X}\{ f(\zeta) - \frac{1}{2\epsilon} \|x-\zeta\|^2\} \text{ for } x\in X.$$
Let $\delta = \sqrt{4\epsilon M}$ and 
$$ X_{\delta} = \{ x : B_{\delta}(x)\subset X\}$$
Then
\begin{enumerate}
\item $f^{\epsilon}$ is $\frac{1}{2\epsilon}$-semiconvex
\item $f^{\epsilon}\searrow f$ pointwise on $X$ as $\epsilon\to 0$.  In fact
\begin{equation} f^\epsilon(x) \le \sup_{\|\zeta-x\|<\delta} \{f(\zeta) - \frac{1}{2\epsilon} \|x-\zeta\|^2 \} \text{ for } x\in X_{\delta}.\label{eq:supconvolutionupperbound} \end{equation}
\item $f^{\epsilon}$ is $F$-subharmonic on $X_{\delta}$.
\end{enumerate}
\end{lemma}
\begin{proof}
This is \cite[Thm 8.2]{HL_Dirichletduality} (we observe that in that paper $F$ is assumed to depend only on the Hessian part, but the proof works for any constant-coefficient primitive subequation $F$ that has the Negativity property.)  For convenience we give the details. (1) follows from Lemma \ref{lem:supconvolutionsemiconvex}.    For (2) observe that clearly we always have $f^\epsilon\ge f$, and if $x\in X_{\delta}$ and $\|x-\zeta\|>\delta$ then $f(\zeta) - f(x) - \frac{1}{2\epsilon} \|x-\zeta\|^2 \le 2M -\frac{\delta^2}{2\epsilon} =0$, so $f(\zeta)-  \frac{1}{2\epsilon} \|x-\zeta\|^2 \le f(x)$, which proves \eqref{eq:supconvolutionupperbound}.    By upper-semicontinuity of $f$, this implies $f^{\epsilon}$ decreases pointwise to $f$ as $\epsilon$ tends to zero.

For (3) make a change of variables $\tau = \zeta -x$ to get
$$f^{\epsilon}(x) = \sup_{\|\tau\|<\delta} \{ f(x+\tau) - \frac{1}{2\epsilon} \|\tau\|^2\} \text{ for } x\in X_{\delta}.$$
The translation $x\mapsto f(x-\tau)$ is $F$-subharmonic (Proposition \ref{prop:basicproperties}(5)), and by the Negativity property of $F$ so is the function $x\mapsto f(x-\tau) - \frac{1}{2\epsilon} \|\tau\|^2$.  Now $f^{\epsilon}$ is semiconvex, so continuous, so we may replace the supremum with its upper-semicontinuous regularisation.  Thus (Proposition \ref{prop:basicproperties}(4)) $f^{\epsilon}$ is $F$-subharmonic on $X_\delta$ as claimed.
\end{proof}

Thus the sup-convolution allows one to approximate bounded $F$-subharmonic functions by semiconvex ones (albeit on a slightly smaller subset).  For this reason the following terminology will be useful:

\begin{definition}[Eventually on compact sets]
Suppose $P$ is a property of functions defined on open subsets of $\Omega$, and let $f_j$ be a sequence of functions $f_j$ defined on $\Omega$.
\begin{enumerate}
\item We say that $P$ holds \emph{eventually on relatively compact subsets} if given any open $K\Subset \Omega$ there is a $j_0$ such that $f_j|_{K}$ has property $P$ for all $j\ge j_0$.
\item We say $f_j\searrow f$ \emph{pointwise eventually on relatively compact subsets} as $j\to \infty$ if for any open $K\Subset \Omega$ the sequence $f_j|_{K}$ decreases pointwise to $f|_{K}$ once $j$ is sufficiently large.
\end{enumerate}
\end{definition}

\begin{proposition}[Exhaustive approximation by semiconvex functions]\label{prop:approxsemi}
Let $F\subset J^2(X)$ be a constant-coefficient primitive subequation with the Negativity Property.

Suppose that $f:X\to\mathbb R$ is $F$-subharmonic and bounded from below.    Then there exists a sequence of functions 
$$ f_j : X\to \mathbb R$$ such that
\begin{enumerate}
%\item $f_j$ is eventually semiconvex on relatively compact subsets as $j\to \infty$
\item $f_j$ is semiconvex on $X$.
\item $f_j \searrow f$ pointwise eventually on relatively compact subsets as $j\to \infty$.
%$$ \lim_{j\to \infty, j\ge j_0} f_j(x) \searrow f(x) \text{ for all } x\in K.$$
%\item $f_j\ge f$ eventually on relatively compact sets as $j\to \infty$.
%.  That is given any open $K\subset\subset X$ there is a $j_0$ such that $f_j\ge f$ on $K$ for all $j\ge j_0$.
\item $f_j$ is $F$-subharmonic eventually on relatively compact subsets.
%.  That is, given any open $K\subset\subset X$ there is a $j_0$ such that $f_j$ is $F$-subharmonic on $K$ for all $j\ge j_0$.
\end{enumerate}
\end{proposition}
\begin{proof}
Let $K_1\Subset K_2\Subset K_3\Subset \cdots$ be open and each relatively compact in the next, with $\Omega = \cup_j K_j$.    By assumption $f$ is bounded from below on $X$, and since it is also upper-semicontinuous
$$ M_j : = \sup_{x\in K_j} |f(x)|$$
is finite.

Now for $\epsilon>0$ consider
$$ f_{j,\epsilon}(x) : = \sup_{\zeta \in K_j} \{ f(\zeta) - \frac{1}{2\epsilon} \|x-\zeta\|^2\}.$$
Lemma \ref{lem:supconvolutionsemiconvex} implies that $f_{j,\epsilon}$ is semiconvex on $X$. \\

{\bf Claim: } Given $j\ge 2$ and $\epsilon_j>0$ there exists an $\epsilon_{j+1}<\epsilon_j$ such that
\begin{enumerate}[(a)]
\item $f_{j+1,\epsilon_{j+1}} < f_{j,\epsilon_j} \text{ on } K_{j-1}$.
\item $f_{j+1,\epsilon_{j+1}}(x)\le \sup_{|x-\zeta|<1/j} f(\zeta) \text{ for all } x\in K_{j}.$
\item $f_{j+1,\epsilon_{j+1}} \text{ is }F\text{-subharmonic on } K_j$.
\end{enumerate}
To see this, let $M=\inf_{\zeta\in K_{j-1}} |f_{j,\epsilon_j}(\zeta)|$ (which is finite as $f_{j,\epsilon_j}$ is semiconvex, and so in particular continuous, on $\overline{K_{j-1}}$).   Any $\zeta\in K_{j+1}\setminus K_{j}$ is a bounded distance away from any $x\in K_{j-1}$, say $\|\zeta-x\|>\delta'>0$.    So for such $\zeta$ and sufficiently small $\epsilon_{j+1}<\epsilon_j$,
$$ f(\zeta) - \frac{1}{2\epsilon_{j+1}} \|x-\zeta\|^2 \le M_{j+1} - \frac{1}{2\epsilon_{j+1}} \delta'^2 < M\le f_{j,\epsilon_j}(x).$$
  On the other hand for $\zeta\in {K}_j$ and $x\in K_{j-1}$, clearly
$$f(\zeta) - \frac{1}{2\epsilon_{j+1}} \|x-\zeta\|^2\le  f(\zeta) - \frac{1}{2\epsilon_{j}}\|x-\zeta\|^2 \le f_{j,\epsilon_j}(x).$$
Putting these together,
$$ f(\zeta) - \frac{1}{2\epsilon_{j+1}} \|x-\zeta\|^2\le   f_{j,\epsilon_j}(x) \text{ for all } \zeta\in {K}_{j+1}, x\in K_{j-1}$$
and taking the supremum over all $\zeta\in K_{j+1}$ proves (a).  Furthermore there is no loss in assuming that $\epsilon_{j+1}$ is chosen small enough so that setting $\delta:=\sqrt{4M_{j+1}\epsilon_{j+1}}$ we have $K_j\subset \{ x : B_{\delta}(x) \subset K_{j+1}\}$ and $\delta<1/j$.  Thus Lemma \ref{lem:supconvolution}(2,3) imply (b,c) respectively.\\
% \eqref{cor:supconvolutionclaim2}.

Now let $\epsilon_1=\epsilon_2:=1$ and $\epsilon_3>\epsilon_4>\cdots$ be defined inductively so the claim holds for all $j$ and set $f_j:=f_{j,\epsilon_j}$.  Condition (1) in the statement of the theorem holds by construction, and since any open $K\Subset X$ is eventually contained in $K_j$ for $j$ sufficiently large, conditions (a,b,c) imply (2) and (3). 

For the final statement in the complex case, we may take each of the $K_j$ to be $G$-invariant, at which point it is clear that $f_{j,\epsilon}$ is also $G$-invariant.
\end{proof}

\begin{corollary}[Invariance] \label{cor:prop:approxsemiTinvariant}
In the setting of Proposition \ref{prop:approxsemi} assume that $X\subset \mathbb R^{2n}\simeq \mathbb C^n$ and that $G$ is a subgroup of the group of $n\times n$ unitary matrices.  Assume also that both $X$ and $f$ are $G$-invariant.  Then the functions $f_j$ can be taken to be $G$-invariant as well.
\end{corollary}
\begin{proof}
In the above proof we can take each of the $K_j$ to be $G$-invariant, and then the sup-convolutions $f_{j,\epsilon}$ will also be $G$-invariant.
\end{proof}

\subsection{Smooth Mollification}\label{sec:mollification}

We next show that the smooth mollification of a continuous $F$-subharmonic function is again $F$-subharmonic, as long as we assume that $F$ is convex.  Fix a mollifying function $\phi$ on $\mathbb R^n$, so $\phi$ is smooth, non-negative, compactly supported with $\int_{\mathbb R^n} \phi(t) dt =1$ and so $\phi_{\epsilon}(t): = \epsilon^{-n} \phi(\epsilon^{-1}t)$ tends to a dirac delta as $\epsilon\to 0$.

\begin{proposition}\label{prop:convolution}[Smooth Mollifications Remain $F$-subharmonic]
Let $F$ be a constant coefficient convex primitive subequation, and $f$ be continuous and $F$-subharmonic on $X$.    Let $K\Subset \Omega$ be open, and for sufficiently small $\epsilon$ define
$$f_{\epsilon}(x) : = f* \phi_{\epsilon}(x) = \int_{\mathbb R^n} \phi_{\epsilon}(t) f(x-t) dt.$$
Then $f_{\epsilon}$ is $F$-subharmonic on $K$ for all $\epsilon$ sufficiently small.
\end{proposition}
\begin{proof}
Note that as $\phi$ is compactly supported, $f_{\epsilon}$ is well-defined for $\epsilon$ sufficiently small.  We may think of the integral in the convolution as a limit of Riemann sums, each of the form
$$ \sum_{i=1}^N \frac{\phi_\epsilon(t_i)}{N} f(x-t_i)$$
for some points $\{ t_i\}$ (where the $t_i$ of course also depend on $N$ which is dropped from notation).  Note each function $x\mapsto  f(x-t_i)$ is a translation of $f$, and hence is $F$-subharmonic on $K$.   Let
$$ S_N : = \sum_{i=1}^N \frac{\phi_\epsilon(t_i)}{N}$$
Then $S_N\to \int_{\mathbb R^n} \phi_\epsilon(dt) =1$ as $N\to \infty$.  Now as $F$ is convex, the convex combination

$$ h_N(x):= \sum_{i=1}^N \frac{\phi_\epsilon(t_i)}{N S_N} f(x-t_i)$$
is $F$-subharmonic on $K$ (Proposition \ref{prop:convexcombintation}).  And $h_N$ tends to $f_\epsilon$ locally uniformly as $N$ tends to infinity, so $f_{\epsilon}$ is $F$-subharmonic as well.
\end{proof}

\begin{proposition}[Approximation by smooth functions]\label{prop:smoothapproximation}
Let $F\subset J^2(X)$ be a constant-coefficient convex primitive subequation with the Negativity Property.
Suppose that $f$ is $F$-subharmonic on $X$ and continuous.    Then there exists a sequence of functions 
$$ h_j : X\to \mathbb R$$ with the following properties
\begin{enumerate}
\item $h_j$ is smooth.
\item $h_j\to f$ locally uniformly.
%$$ \lim_{j\to \infty, j\ge j_0} f_j(x) \searrow f(x) \text{ for all } x\in K.$$
%\item $f_j\ge f$ eventually on compact sets as $j\to \infty$.
%.  That is given any open $K\subset\subset X$ there is a $j_0$ such that $f_j\ge f$ on $K$ for all $j\ge j_0$.
\item $h_j$ is $F$-subharmonic eventually on relatively compact subsets.
%.  That is, given any open $K\subset\subset X$ there is a $j_0$ such that $f_j$ is $F$-subharmonic on $K$ for all $j\ge j_0$.
\end{enumerate}
\end{proposition}
\begin{proof}
Let $K_1\Subset K_2\Subset \cdots \Subset X$ be an exhausting family of open subsets, each relatively compact in the next. Using Proposition \ref{prop:convolution} may choose $\epsilon_j>0$ such that the function $h_j:= f* \phi_{\epsilon_j}$ is $F$-subharmonic on $K_j$.  We can even arrange so that $h_j$ is smooth on $\overline{K}_j$ and then extend it arbitrarily to a smooth function on all of $X$.  Further we may as well assume that $\epsilon_j\to 0$ as $j\to \infty$.  Properties (1,2,3) then follow as an open $K\Subset X$ is contained in $K_i$ for $i$ sufficiently large.
\end{proof}

In the complex case we will want to apply this in a way that keeps the property of being $\mathbb T$-invariant.   This can be done with the same kind of mollification argument using polar coordinates in the second variable, as long as we work away from the second variable being zero.  The following statement is sufficient for our needs.

\begin{proposition}[Invariance]\label{prop:smoothapproximation_complex}
Suppose $X\subset \mathbb C^n$ is open and that $F\subset J^{2,\mathbb C}(X)$ is a complex constant-coefficient convex primitive subequation with the Negativity Property.   Suppose also
$$\Omega\subset X\times \mathbb C^*$$
is open, $\mathbb T$-invariant and that $f:\Omega\to \mathbb R$ is continuous, $F\#_{\mathbb C}\calPos^{\mathbb C}$-subharmonic and $\mathbb T$-invariant.

Then there exists a sequence of smooth $\mathbb T$-invariant functions $h_j:\Omega\to \mathbb R$ that are eventually 
$F\#_{\mathbb C}\calPos^{\mathbb C}$-subharmonic on compact sets, and converge locally uniformly to $f$ on $\Omega$.
\end{proposition}
\begin{proof}
Exhaust $\Omega$ by open $K_1\Subset K_2\Subset \cdots$ that are each $\mathbb T$-invariant.  Consider $\phi(z,w) = (z,e^w)$ and let $g(z,w) = f(z,e^w)$.  As $f$ is $\mathbb T$-invariant we have that $g$ is independent of the imaginary part of $w$ and is $F\#_{\mathbb C}\calPos^{\mathbb C}$-subharmonic by Lemma \ref{lem:compositionbiholomorphism}.     For each $j$ let $\epsilon$ be sufficiently small so that the mollification
$$g_j(z,w):= \int_{\mathbb C^n} \int_{\mathbb R} g(z,w-t) \phi_{\epsilon}(z) \phi_\epsilon(t) dt dz$$
is well defined, smooth and  $F\#_{\mathbb C}\calPos^{\mathbb C}$-subharmonic on $\phi^{-1}(K_j)$.    Notice that $g_j$ is independent of the imaginary part of the second variable, and just as in the previous proof we may assume $g_j$ is defined, smooth and independent of the imaginary part on all of $\phi^{-1}(\Omega)$.  

Thus $g_j$ decends to a function $h_j$ on $\Omega$ (i.e. $g_j(z,w) = h_j(z,e^w)$) and the $h_j$ have the desired properties.
\end{proof}

\section{Reductions}\label{sec:reductions}

We now consider four reductions concerning the minimum principle that work in both the real and complex case.  The first allows us to reduce the dimension $m$ of the second variable to $1$  (that is, in the real case we can assume $\Omega\subset X\times \mathbb R$ and in the complex case $\Omega\subset X\times \mathbb C$).    In the second reduction we show that if one shows that marginal functions of $F\#\calPos$-subharmonic functions are $F$-subharmonic, then the minimum principle holds for $F$ (i.e.\ the fact that  the projection of an  $F\#\calPos$-pseudoconvex set is $F$-pseudoconvex is then automatic).

Our third reduction shows that it is sufficient to consider $F\#\calPos$-subharmonic functions $f$ that are both bounded below, and whose fibrewise minimum is attained strictly in the interior of its domain (a property that is strictly weaker than the function being exhaustive).   In the fourth we show that we can even assume the function $f$ is semiconvex.

\subsection{Reduction in dimension of the second variable}

\begin{proposition}\label{prop:mcanbeone}
Suppose $F\subset J^2(X)$ is closed, and the minimum principle holds for $F\#\calPos$-subharmonic functions $f$ on $F\#\calPos$-pseudoconvex domains $\Omega\subset X\times \mathbb R$ with connected fibres.  Then the minimum principle holds for $F$ in general.    

The analogous statement holds in the complex case in which we assume $\Omega\subset X\times \mathbb C$ and that both $f$ and $\Omega$ are $\mathbb T$-invariant.
\end{proposition}
\begin{proof}
We need to show the minimum principle holds  for $F\#\calPos$-subharmonic functions on $F\#\calPos$-pseudoconvex domains $\Omega\subset X\times \mathbb R^m$ with connected fibres for any $m\ge 1$.   We will use induction on $m$, the hypothesis of the Proposition being the case $m=1$.

So assume the statement we want holds for all integers up to $m$.    For ease of notation set $$\calPos_n:= \calPos_{\mathbb R^n}.$$ 
From Proposition \ref{prop:associativityofproducts} and Example \ref{sec:exampleI},
$$ F\#\calPos_{m+1} = F\# (\calPos_{m} \#\calPos_1) = (F\# \calPos_{m})\#\calPos_1.$$

Now let $$\Omega\subset X\times \mathbb R^{m+1} = X\times \mathbb R^{m}\times \mathbb R$$ be an $F\#\calPos_{m+1}$-pseudoconvex domain with connected fibres, and let $f$ be $F\#\calPos_{m+1}$-subharmonic on $\Omega$.     We will use $(x,\zeta,y)$ as coordinates on $X\times \mathbb R^{m}\times \mathbb R$ and apply the inductive hypothesis twice to the function $f$, first taking the infimum over $y$ and then the infimum over $\zeta$. %The idea is to think of $f(x,\zeta,y)$ as a function in $(x,\zeta)$ and $y$ that is $(F\#\calPos_{m})\#\calPos_1$-subharmonic, and use the inductive hypothesis twice.

To this end let $\pi_X:X\times \mathbb R^{m+1}\to X$ and $\pi_1: X\times \mathbb R^m \times \mathbb R \to X\times \mathbb R^{m}$ and $\pi_2: X\times \mathbb R^{m}\to X$ be the natural projections, so $\pi_X = \pi_2\pi_1$.  Set
\begin{align*}
 \Omega_{(x,\zeta)} &:= \{ y \in \mathbb R : (x,\zeta,y) \in \Omega\}\\
 g_1(x,\zeta) &:= \inf_{y\in \Omega_{(x,\zeta)}}  f(x,\zeta, y) \text{ for } (x,\zeta)\in \pi_1(\Omega).
 \end{align*}

{\bf Claim: } $\pi_1(\Omega)\subset X\times \mathbb R^m$ is $F\#\calPos_m$-pseudoconvex and $g_1$ is $F\#\calPos_m$-subharmonic. \\

In fact this follows by the inductive hypothesis (with $m=1$).   First we check that for each $(x,\zeta)\in X\times \mathbb R^{m}$ the set $\Omega_{(x,\zeta)}$ is connected.    But this is clear since
$$\Omega_{(x,\zeta)} = \{ y\in \mathbb R : (\zeta, y) \in \Omega_x\}.$$
Now $\Omega_x$ is convex (Corollary \ref{cor:slicesofpseudoconvexareconvex}), which implies that $\Omega_{(x,\zeta)}$ is convex, so certainly connected.  Second, our hypothesis is that $\Omega$ is $F\#\calPos_{m+1} = (F\#\calPos_m)\#\calPos_1$-psuedoconvex and $f$ is  $F\#\calPos_{m+1} = (F\#\calPos_m)\#\calPos_1$-subharmonic.   Thus the inductive hypothesis applies to give the claim.\\

Now let 
$$ g_2(x) := \inf_{\zeta\in \pi_1(\Omega)_x} g_1(x,\zeta) \text{ for } x\in \pi_X(\Omega).$$

We claim the inductive hypothesis (for $m$) implies $\pi_X(\Omega) =\pi_2(\pi_1(\Omega))$ is $F$-pseudoconvex and $g_2$ is $F$-subharmonic.   For this we need only verify that $\pi_1(\Omega)_x$ is connected.  But this follows easily as
$$\pi_1(\Omega)_x = \{ \zeta : (x,\zeta)\in \pi_1(\Omega) \} = \{ \zeta : (\zeta,y)\in \Omega_x \text{ for some } y\}$$
which is convex as $\Omega_x$ is convex.  

Now elementary considerations show that
$$ g_2(x) =  \inf_{\zeta\in \pi_1(\Omega)_x}  \inf_{y\in \Omega_{(x,\zeta)}}  f(x,\zeta, y) = \inf_{(\zeta,y) \in \Omega_x} f(x,\zeta, y) \text{ for } x\in \pi_X(\Omega)$$
which is the marginal function of $f$.    Thus the statement we want also holds for integers up to $m+1$, and the induction is complete.
\end{proof}

\subsection{Projections are automatically pseudoconvex}

\begin{proposition}\label{prop:functiononly}
Let $F\subset J^2(X)$ be a primitive subequation.    Suppose that for any $F\#\calPos$-pseudoconvex domain $\Omega\subset X\times \mathbb R$ with connected fibres and any $f:\Omega\to \mathbb R\cup \{-\infty\}$ that is not identically $-\infty$ it holds that
$$ f \text{ is } F\#\calPos\text{-subharmonic on }\Omega \Rightarrow g(x): = \inf_{y\in \Omega_x} f(x,y) \text{ is } F\text{-subharmonic on }\pi(\Omega).$$
 Then $\pi(\Omega)$ is automatically $F$-pseudoconvex, and so the minimum principle holds for $F$.

The analogous statement holds in the complex case, if we make the additional hypothesis that $\Omega\subset X\times \mathbb C$ and $f$ are $\mathbb T$-invariant.
\end{proposition}

\begin{proof}
We start with the real case.  Suppose that $\Omega\subset X\times \mathbb R$ is $F\#\calPos$-pseudoconvex with connected fibres, and we aim to show $\pi(\Omega)$ is $F$-subharmonic.    By hypothesis, there is a continuous exhaustive  $F\#\calPos$-subharmonic function $f:\Omega\to \mathbb R$.  So by hypothesis its marginal function $g(x): = \inf_{y\in \Omega_x} f(x,y)$ is $F$-subharmonic on $\pi(\Omega)$.  But $g$ is both continuous and exhaustive (Lemma \ref{lem:marginalsareexhausting}) from which we conclude that $\pi(\Omega)$ is $F$-pseudoconvex.  This proves the minimum principle holds when $m=1$.  But this is sufficient to prove the minimum principle for all $m$ by Proposition \ref{prop:mcanbeone}, thereby completing the proof in the real case.

In the complex case, the fact that $\pi(\Omega)$ is again $F$-subharmonic follows in the same way once it is established that any $\mathbb T$-invariant $F\#_{\mathbb C}\calPos^{\mathbb C}$-pseudoconvex domain admits a continuous exhaustive  $F\#_{\mathbb C}\calPos^{\mathbb C}$-subharmonic function that is $\mathbb T$-invariant, which we do in Lemma \ref{lem:exhaustionindependent}.
\end{proof}

\begin{lemma}[Marginals functions preserving continuity and being exhaustive]\label{lem:marginalsareexhausting}
Let $\Omega\subset X\times \mathbb R$ and $f:\Omega\to \mathbb R$ be continuous (resp.\ exhaustive).  Then $g(x): = \inf_{y\in \Omega_x} f(x,y)$ is continuous (resp.\ exhaustive).
\end{lemma}
\begin{proof}
For any real number $a$ we have $\{ x\in \pi_X(\Omega) : g(x) <a\} \subset \pi_X ( \{ (x,y)\in \Omega : f(x,y)<a\}$.    If $f$ is exhaustive then $\{ (x,y)\in \Omega : f(x,y)<a\}$ is relatively compact, and hence so is $\{ x\in \pi_X(\Omega) : g(x) <a\}$, proving that $g$ is exhaustive.

Now assume $f$ is continuous. If $g(x)<a$ for some $x$ then there is a $y\in \Omega_{x}$ such that $f(x,y)<a$.  Then $f<a$ on some neighbourhood $U_0\times U_1$ of $(x_0,y_0)$, and hence $g<a$ on $U_0$.     Thus $g$ is upper-semicontinuous.

Next suppose $g(x_n)<a$ for a sequence of points $x_n\in \pi_X(\Omega)$ that converge to some $x\in \pi_X(\Omega)$ as $n$ tends to infinity.  For each $n$ there is a $y_n$ with $f(x_n,y_n)<a$.   Thus $(x_n,y_n)$ lie in the relatively compact set $\{(x,y) : f(x,y)<a\}$, and so after taking a subsequence we may assume it converges to some point $(x,y)\in \Omega$.  Continuity of $f$ yields $f(x,y)\le a$, so $g(x)\le a$.  Thus $g$ is also lower-semicontinuous, and hence continuous.
\end{proof}

\begin{lemma}[Existence of $\mathbb T^m$-invariant exhaustive functions]\label{lem:exhaustionindependent}
Let $F\subset J^{2,\mathbb C}(X)$ be a complex primitive subequation and $\Omega\subset X\times \mathbb C^m$ be a $F\#_{\mathbb C}\calPos^{\mathbb C}$-pseudoconvex and $\mathbb T^m$-invariant domain.    Then there exists a continuous exhaustive $F\#_{\mathbb C}\calPos^{\mathbb C}$-function on $\Omega$ that is $\mathbb T^m$-invariant.
\end{lemma}
\begin{proof}
As $\Omega$ is  $F\#_{\mathbb C}\calPos_{\mathbb C}$-pseudoconvex there exists an exhaustive continuous $F\#_{\mathbb C}\calPos^{\mathbb C}$-subharmonic function $u$ on $\Omega$.   For $e^{i\theta}\in \mathbb T^m$ set
$$ u_{\theta}(z,w): = u(z,e^{i\theta}w).$$
which by Lemma \ref{lem:twisttheta} is $F\#_{\mathbb C}\calPos^{\mathbb C}$-subharmonic on $\Omega$.

Now set
$$ v(x,y) := \sup_{e^{i\theta}\in \mathbb T^m} u_{\theta}(x,y).$$
Clearly $v\ge u$, and so is exhaustive as $u$ is.  We claim that $v$ is continuous.  Let $\epsilon>0$.   By continuity of $u$ and compactness of $\mathbb T^m$ there is a $\delta>0$ such that if $\|(z,w)-(z',w')\|<\delta$ then $|u_{\theta}(z,w) - u_{\theta}(z',w')|<\epsilon$ for all $e^{i\theta}\in \mathbb T^m$.  Let  $\|(z,w)-(z',w')\|<\delta$ .  There is a $e^{i\theta_0}\in \mathbb T^m$ such that $v(z,w) < u_{\theta_0}(z,w)  + \epsilon < u_{\theta_0}(z',w') + 2\epsilon < v(z',w') + 2\epsilon$.  Swapping the role of $(z,w)$ and $(z',w')$ gives $|v(z,w)-v(z',w')|<2\epsilon$, proving continuity of $v$.

Hence $v$ is locally bounded above and equal to its upper semicontinuous regularisation.  So Proposition \ref{prop:basicproperties}(4) tells us that $v$ is $F\#_{\mathbb C}\calPos_{\mathbb C}$-subharmonic on $\Omega$.  Note that $v$ is $\mathbb T^m$-invariant by construction.

\end{proof}

\begin{lemma}\label{lem:twisttheta}
Let $f$ be an $(F\#_{\mathbb C}\calPos^{\mathbb C})$-subharmonic on a $\mathbb T^m$-invariant open $\Omega\subset X\times \mathbb C^m$.  Then for any fixed $e^{i\theta}\in \mathbb T^m$ the function $f_{\theta}(z,w):= f(z,e^{i\theta} w)$ is $(F\#_{\mathbb C}\calPos^{\mathbb C})$-subharmonic on $\Omega$
\end{lemma}
\begin{proof}
Apply Lemma \ref{lem:compositionbiholomorphism} to the biholomorphism $\zeta(w) := e^{i\theta} w$.
\end{proof}

\subsection{Reduction to functions that are relatively exhaustive}

We next show how to reduce to functions that are bounded from below, and have the property that on each fibre the minimum is attained strictly away from the boundary.  In fact we will need two versions of this, and the following terminology is useful.   Let $\Omega'\subset U\times \mathbb R^m$ where $U\subset \mathbb R^n$ is open and  $f:\Omega\to \mathbb R\cup \{-\infty\}$.

\begin{definition}[Relatively exhaustive]
We say $f$ is \emph{relatively exhaustive} if for all open $V\Subset U$ and all real $a$ the set
$$ \{ (x,y)\in \Omega : f(x,y)<a \text{ and } x\in V\}$$
is relatively compact in $\Omega'$.
\end{definition}

\begin{definition}[Functions with fibrewise minimum strictly in the interior]\label{def:fibrewisemin}
 We say $f$  \emph{attains its fiberwise minimum strictly in the interior} if for any $x_0\in U$ there exists a real number $a$ and neighbourhood $x_0\in V\subset U$ such that
$$ K_V:= \{ (x,y)\in \Omega' : f(x,y)<a \text{ and } x\in V\}$$
is relatively compact in $\Omega'$ and $\pi(K_V) = V$.

(In other words, the marginal function of $f$ is bounded from above by $a$ on $V$, and the set of points $(x,y)$ over $V$ for which $f$ is less than $a$ is contained strictly away from the boundary of $\Omega'$.)
\end{definition}

\begin{lemma}\label{lem:fibrewise}\label{we may be able to get rid of this lemma}
If $f$ is upper-semicontinuous and relatively exhaustive then is attains its fiberwise minimum strictly in the interior.
\end{lemma}
\begin{proof}
Let $g(x) = \inf_{y\in \Omega_x} f(x,y)$.  For any $x_0\in \pi(\Omega)$ pick any $a$ with $a> g(x_0)$.  Then there is a neighbourhood $V$ such that $g<a$ on $V$.  The set $K_V$ is relatively compact in $\Omega'$ as $f$ is exhaustive, and $\pi(K_V) = V$ by construction.
\end{proof}  

We will prove the next two statements simultaneously.

\begin{proposition}\label{prop:reductionfibrewiseminimum}
Let $F\subset J^2(X)$ be a primitive subequation that is constant coefficient and has the Negativity Property.     Suppose for any $$ f:\Omega'\to \mathbb R$$
such that
\begin{enumerate}[(a)]
\item $\Omega'\subset X\times \mathbb R$ is open and has connected fibres,
\item There exists a bounded $F$-subharmonic function on $\pi(\Omega')$,
\item $f$ is $F\#\calPos$-subharmonic,
\item $f$ is bounded from below,
\item $f$ attains its fiberwise minimum strictly in the interior,
\end{enumerate}
the marginal function $$g(x) = \inf_{y\in \Omega_x} f(x,y)$$ is $F$-subharmonic on $\pi(\Omega')$.   (Note that in the hypothesis we are not assuming that $\Omega'$ is $F\#\calPos$-pseudoconvex).

Then the minimum principle holds for $F$.   The analogous statement holds in the complex case under the additional assumption that $f$ and $\Omega'\subset X\times \mathbb C$ are $\mathbb T$-invariant.
\end{proposition}

\begin{proposition}\label{prop:reductionlemmarelativelyexhausting}
Let $F\subset J^2(X)$ be a constant-coefficient primitive subequation that has the Negativity Property.    Then the conclusion of Proposition \ref{prop:reductionfibrewiseminimum} continues to hold if condition (e) is replaced with
\begin{enumerate}[(e')]%\setcounter{enumi}{5}
%\item $\Omega'\subset \pi^{-1}(U)$
%\item $f$ is $F\#\calPos$-subharmonic 
%\item $f$ attains its fiberwise minimum in the interior
\item $f$ is relatively exhaustive.
\end{enumerate}
\end{proposition}

\begin{proof}
Using Lemma \ref{lem:fibrewise} it is clear that Proposition \ref{prop:reductionlemmarelativelyexhausting} implies Proposition \ref{prop:reductionfibrewiseminimum}, so it is sufficient to prove only the former. 

Let $\Omega\subset X\times \mathbb R$ be an $F\#\calPos$-pseudoconvex domain with connected fibres and $\tilde{f}:\Omega\to \mathbb R\cup \{-\infty\}$ be $F\#\calPos$-subharmonic and not identically $-\infty$ with marginal function $$\tilde{g}(x):= \inf_{y\in \Omega_x} f(x,y).$$     We claim that $\tilde{g}$ is $F$-subharmonic on $\pi(\Omega)$.  By Proposition \ref{prop:functiononly} this implies the minimum principle holds for $F$.

By the hypothesis that $\Omega$ is $F\#\calPos$-pseudoconvex, there exists an exhaustive $F\#\calPos$-subharmonic  function $u$ on $\Omega$.    By Proposition \ref{prop:approxboundedbelow} there is an open cover of $\pi(\Omega)$ by open subsets $U\subset \pi(\Omega)$ on which there exist an $v\in F(U)$ that is bounded from below.   In such a case consider for each $j\in \mathbb N$ the function
$$ f_j := \max\{ \tilde{f}, \pi^*v -j , u-j\} \text{ on } \Omega' := \Omega\cap \pi^{-1}(U).$$
Observe that both $\pi^*v - j$ and $u-j$ are $F\#\calPos$-subharmonic (the first of these is Proposition \ref{prop:approxboundedbelow}, and the second follows as $F\#\calPos$ has the Negativity Property) and hence so is $f_j$.  Thus $(\Omega',f_j)$ and satisfy properties (a--d,e'), so by hypothesis $$g_j(x):= \inf_{y\in \Omega_x} f_j(x,y)$$ is in $F(U)$.  But $f_j$ decreases to $\tilde{f}|_{\Omega'}$ as $j\to \infty$, and so $g_j$ decreases pointwise to $\tilde{g}|_U$ as $j\to \infty$.  Thus $\tilde{g}|_U$ is $F$-subharmonic, and since this holds for all such $U$ we conclude that $\tilde{g}$ is $F$-subharmonic on $\pi(\Omega)$ as claimed.

The complex case is the same since we can pick $u$ to be $\mathbb T$-invariant by Lemma \ref{lem:exhaustionindependent}, so both $\Omega'$ and $f_j$ are $\mathbb T$-invariant if $\tilde{f}$ and $\Omega$ are.
\end{proof}

\begin{remark}
\begin{enumerate}
\item It is possible to prove the minimum principle in the real case only using Proposition \ref{prop:reductionlemmarelativelyexhausting}.  However we will instead use Proposition \ref{prop:reductionfibrewiseminimum} so that we can more easily reuse our statements when proving the complex case.
\item
 The only place so far we have used that $F$ is constant-coefficient is in ensuring that $X$ is covered by open sets that admit $F$-subharmonic functions bounded from below.     There is a a version of Proposition \ref{prop:reductionlemmarelativelyexhausting} that does not require $F$ be constant-coefficient if one removes hypotheses (b) and (d).  The proof is essentially the same, and is left to the reader. 
\end{enumerate}
\end{remark}

\begin{comment}
\begin{proposition}\label{prop:fibrewisemin}
Let $F\subset J^2(X)$ be a primitive subequation that has the Negativity Property.    Then the conclusion of Proposition \ref{prop:reductionlemmarelativelyexhausting} continues to hold if we replace condition (d) with
\begin{enumerate}[(d')]
%\item $\Omega'\subset \pi^{-1}(U)$
%\item $f$ is $F\#\calPos$-subharmonic 
\item $f$ is attains its fiberwise minimum in the interior.
\end{enumerate}
\end{proposition}
\end{comment}

\section{The Minimum Principle in the Real Convex Case}\label{sec:minimumprincipleconvex}

We now use approximation arguments to reduce the minimum principle to considering first only semiconvex funtions, and then only smooth functions.  Note that we are restricting attention for the moment to the real case;   the complex case is slightly more involved due to the requirement that all quantities involved be $\mathbb T$-invariant, and will be taken up in the next section.

\subsection{Reduction to semiconvex functions}

\begin{proposition} \label{prop:reductionsemiconvex} Let $F\subset J^2(X)$ be a constant-coefficient primitive subequation that has the Negativity Property.   Then the real case of Proposition \ref{prop:reductionfibrewiseminimum} continues to hold we assume in addition that
\begin{enumerate}[(f)]%\setcounter{enumi}{5}
%\item $\Omega'\subset \pi^{-1}(U)$
%\item $f$ is $F\#\calPos$-subharmonic 
%\item $f$ attains its fiberwise minimum in the interior
\item $f$ is semiconvex.
\end{enumerate}
\end{proposition}

\begin{proof}
Let $f:\Omega'\to \mathbb R$ be such that (a--e) hold, and as usual set $g(x) = \inf_{y\in \Omega_x'} f(x,y)$.  We will show $g$ is $F$-subharmonic on $\pi(\Omega')$, so Proposition \ref{prop:reductionfibrewiseminimum} applies.

Since $f$ satisfies condition (d) (namely $f$ is bounded from below) we can use the sup-convolution to approximate $f$ by $F\#\calPos$-subharmonic functions that are semiconvex.   In detail, let $$ {f}_j : \Omega'\to \mathbb R$$ be the sequence of such approximations of $f$ provided by Proposition \ref{prop:approxsemi}, and  for convenience of the reader we recall these functions satisfy:
\begin{enumerate}[(1)]
%\item Each $\hat{f}_j$ is eventually smooth on relatively compact sets as $j\to \infty$.
\item ${f}_j\searrow f$ eventually on relatively compact sets  as $j\to \infty$.
%\item $\hat{f}_j\ge f$ eventually on compact sets of $\Omega$
\item ${f}_j$ is semiconvex on $\Omega'$.
\item ${f}_j$ is $F\#\calPos$-subharmonic eventually on relatively compact sets of $\Omega'$.
\end{enumerate}

It is sufficient to prove that $g$ is $F$-subharmonic on some neighbourhood of an arbitrary $x_0\in \pi(\Omega')$.   As $f$ satisfies condition (e) (namely that it attains its relative minimum strictly in the interior) we know there is an $a>g(x_0)$ and a neighbourhood $V$ of $x_0$ in $\pi(\Omega')$ such that
$$K_V := \{ (x,y) \in \Omega' : f(x,y)<a \text{ and } x\in V\}\Subset \Omega'$$
and $\pi(K_V) = V$.   Since $\Omega'$ has connected fibres, we can fix an open $K'$ with $K_V\subset K'\subset \Omega'$ each relatively compact in the next such that $K'$ has connected fibres.  

   %Let $K'$ be a relatively compact neighbourhood of $\overline{K}$ in $\Omega'$.   %Thus
%$$(x,y) \in K'\setminus \overline{K} \text{ and } x\in V \Rightarrow f(x,y)\ge a.$$
%\end{enumerate}
Now choose $j_0$ large enough such that for all $j\ge j_0$,
\begin{enumerate}[(i)]
\item $f_j(x_0,y_0)<a$.
\item $f_j\searrow f$ pointwise on $K'$.
\item $f_j$ is $F\#\calPos\text{-subharmonic and semiconvex on }K'$.
\end{enumerate}
By (i) and continuity of $f_{j_0}$ there are small neighbourhoods $x_0\in U_0\subset \pi(K_V)$ and $y_0\in U_1\subset \Omega'_{x_0}$ such that $U_0\times U_1\subset K$ and
$$ f_{j_0}<a \text{ on } U_0\times U_1.$$
Shrinking $U_0$ if necessary, we may as well assume that there exists a bounded $F$-subharmonic function on $U_0$ (Lemma \ref{lem:boundedbelow}).\\

Now set $K'_{U_0}: = K'\cap\pi^{-1}(U_0)$ and
$$\tilde{f}_j : = f_j|_{K'_{U_0}}.$$

\noindent {\bf Claim I:} If $x\in U_0$ then
\begin{equation} \inf_{y\in K'_x} \tilde{f}_j(x,y) \searrow \inf_{y\in \Omega_x'} f(x,y)=g(x) \text{ as } j\to \infty.\label{eq:claimdecreasecontinuous}\end{equation}
To prove this note first by (ii) we have
$$ \inf_{y\in K'_x} \tilde{f}_j(x,y) \searrow \inf_{y\in K'_x} f(x,y) \text{ as } j\to \infty.$$
On the other hand, (ii) also implies $f_j$ is pointwise decreasing on $U_0\times U_1\subset K$, so 
\begin{equation} f\le f_{j}<f_{j_0}<a \text{ on } U_0\times U_1 \text{ for all } j\ge j_0.\label{eq:lessthanonproduct}\end{equation}
Thus
$$ \inf_{y\in \Omega_x'} f(x,y) \le \inf_{y\in U_1} f(x,y) \le a.$$
But by construction if $y\in \Omega_x'\setminus K'$ then $f(x,y)\ge a$.  So
$$ \inf_{y\in \Omega_x'} f(x,y) = \inf_{y\in K'_x} \tilde{f}_j(x,y)$$
proving Claim I.\\

\noindent {\bf Claim II:} For $j\ge j_0$ the function $\tilde{f}_j:K'_{U_0}\to \mathbb R$ satisfies conditions (a--f).\\

We have arranged that (a) holds (namely that $K'_x$ is connected) and also that $\pi(K'_{U_0}) = U_0$ which admits a bounded $F$-subharmonic function, giving (b).    Conditions (c) and (f) (namely that $\tilde{f}_j$ is $F\#\calPos$-subharmonic and semiconvex on $K_{U_0}$) are given by (iii).  In fact $f_j$ is semiconvex on all of $\Omega'$, so in particular continuous on $\overline{K'}$, and thus $\tilde{f}_j$ is bounded from below, giving (d).    It remains only to verify (e), namely that $\tilde{f}_j$ attains its fibrewise minimum strictly in the interior.    

For this let $\tilde{x}\in U_0$ and fix any neighbourhood $\tilde{x}\in \tilde{V}\Subset U_0$.    Let $a$ be the number used above.   Then 
\begin{align*}
\tilde{K} &: = \{ (x,y)\in K'_{U_0} : \tilde{f}_j(x,y)<a \text{ and } x\in \tilde{V} \} \\
&\subset K_V\cap \pi^{-1}(\tilde{V})
\end{align*}
which implies $\tilde{K}$ is relatively compact in $K'_{U_0}$ (see Lemma \ref{lem:relativelycompacttopological}), and
\eqref{eq:lessthanonproduct} implies $\pi({\tilde{K}})=\tilde{V}$.   This precisely says that $\tilde{f}_j$ satisfies (e), competing the proof of Claim II.\\

So by the hypothesis of the Proposition, the marginal function of $\tilde{f}_j|_{K'_{U_0}}$ is $F$-subharmonic on $U_0$.   That is
$$ \tilde{g}_j(x): = \inf_{y\in K'_x} \tilde{f}_j(x,y) \text{ for } x\in U_0 = \pi(K'_{U_0})$$
is $F$-subharmonic.  But Claim I tells us that $\tilde{g}_j$ decreases pointwise to $g$ on $U_0$ as $j$ tends to infinity, and so $g\in F(U_0)$.   Since $x_0\in \pi(\Omega')$ was arbitrary this implies $g$ is $F$-subharmonic $\pi(\Omega')$, completing the proof.
\end{proof}

\begin{lemma}\label{lem:relativelycompacttopological}
Let $K_V\Subset K'\Subset X\times \mathbb R$ be  open, and  $\tilde{V}\Subset U_0\Subset \pi(K_V)$ also be open.   Assume that $\tilde{K}\subset K_V\cap \pi^{-1}(\tilde{V})$.  Then $\tilde{K}$ is relatively compact in $K'\cap \pi^{-1}(U_0)$.
\end{lemma}
\begin{proof}
Left to the reader.
\end{proof}
%\begin{remark}
%The above proof actually shows that condition (f) can be replaced by the assumption that $f$ is semiconvex.
%\end{remark}

\subsection{Reduction to the smooth case}

\begin{proposition} \label{prop:reductionsmooth} Let $F\subset J^2(X)$ be a constant-coefficient primitive subequation that has the Negativity Property and is convex.     Then the real case of Proposition \ref{prop:reductionsemiconvex} continues to hold if we assume in addition that
\begin{enumerate}[(g)]%\setcounter{enumi}{5}

\item $f$ is smooth.
\end{enumerate}
\end{proposition}

\begin{proof}

The proof of this is almost identical to that of the previous Proposition.    Let $f:\Omega'\to \mathbb R\cup \{-\infty\}$ be such that (a--f) hold, and as usual set $g(x) = \inf_{y\in \Omega_x'} f(x,y)$.   We will show is $F$-subharmonic on $\pi(\Omega')$ so Proposition \ref{prop:reductionsemiconvex} applies.

Since $F$ is now assumed to be convex, so is $F\#\calPos$, and $f$ is assumed to be semiconvex (so in particular continuous), we may use smooth mollification to approximate $f$ locally uniformly by a sequence of smooth functions.    So let $h_j:X\to \mathbb R$ be the sequence of smooth functions furnished by Proposition \ref{prop:convolution}.  Recall these satisfy
\begin{enumerate}[(1)]
%\item Each $\hat{f}_j$ is eventually smooth on relatively compact sets as $j\to \infty$.
\item ${h}_j\to f$ locally uniformly as $j\to \infty$.
%\item $\hat{f}_j\ge f$ eventually on compact sets of $\Omega$
\item ${h}_j$ is smooth.
\item ${h}_j$ is $F\#\calPos$-subharmonic eventually on relatively compact sets as $j\to \infty$.
\end{enumerate}
Then the proof proceeds precisely as in Proposition \ref{prop:reductionsemiconvex}, only we replace \eqref{eq:claimdecreasecontinuous} with the statement that
\begin{equation} \inf_{y\in K'_x} h_j(x,y) \to \inf_{y\in \Omega_x'} f(x,y)=g(x) \text{ locally uniformly as } j\to \infty\label{eq:claimdecreasecontinuousII}\end{equation}
which follows easily from the fact that $h_j\to f$ locally uniformly.\\
\end{proof}

\subsection{The minimum principle in the smooth real case}

We next turn our attention to a statement that guarantees that marginal functions of certain sufficiently smooth $F\#\calPos$-subharmonic functions are $F$-subharmonic.  
%Throughout we let $\pi:\mathbb R^{n}\times \mathbb R^{m}\to \mathbb R^{n}$ be the projection.

\begin{proposition}\label{prop:hessiancalc}
 Let $X\subset \mathbb R^{n}$ be open and $F\subset J^2(X)$ be a primitive subequation.  Let $\Omega\subset X\times \mathbb R^{m}$ be open and assume $f:\Omega\to \mathbb R$ is such that
 \begin{enumerate}[(1)]
%\item   $f$ is $F\#G$-subharmonic
\item   $f$ is $\mathcal{C}^2$.
\item $f$ is strictly convex in the second variable.
\item There exists a $\mathcal C^1$ function $\gamma:X\to \mathbb R^m$ such that
$$ g(x) :=\inf_{y\in \Omega_x} f(x,y) =  f(x,\gamma(x)) \text{ for } x\in \pi(\Omega).$$
%\item   For each fixed $x\in \pi(\Omega)$ the function $y\mapsto f(x,y)$ is strictly convex and attains a unique minimum at some point $w(x)$.
\end{enumerate}
Then 
\begin{enumerate}[(a)]
\item The derivative of $\gamma$ at a point $x\in \pi(\Omega)$ is given by
\begin{equation}\Gamma:=\frac{d\gamma}{dx}  = -D^{-t} C^t\label{eq:hessiancalcderivativew}\end{equation}
where
\begin{equation}\label{eq:hessianblockform} \Hess_{(x,w(x)} (f) = \left( \begin{array}{cc} B&C  \\ C^t & D \\ \end{array}\right)\end{equation}
is the Hessian matrix of $f$ at $(x,\gamma(x))$ in block form (so $B_{ij} = \frac{\partial^2 f}{\partial x_i \partial x_j}|_{(x,\gamma(x))}$ etc.).  
\item The marginal function $g$ is $\mathcal C^2$, and its second order jet at a point $x\in \pi(\Omega)$ is given by
\begin{equation}\label{eq:eq:hessiancalcderivativejetg}
J^2_{x}(g) = i^*_{\Gamma } J^2_{(x,\gamma(x))}(f)
\end{equation}
where $i^*$ is as defined in \eqref{eq:defofistar}. 
\item If additionally $f$ is $F\#\calPos$-subharmonic on $\Omega$ then $g$ is $F$-subharmonic on $\pi(\Omega)$.
\end{enumerate}
\end{proposition}
\begin{proof}
Note first that $D = (\frac{\partial^2 f}{\partial y_i \partial y_j})$ is strictly positive as $f(x,y)$ is assumed to be strictly convex in $y$, so the inverse in \eqref{eq:hessiancalcderivativew} exists.  Now, as $\gamma(x)$ is a minimum of the function $y\mapsto f(x,y)$ we have
\begin{equation}\label{eq:chainruleatmin} \frac{\partial f}{\partial y} (x,\gamma(x)) =0 \text{ for all } x\in X.\end{equation}
Differentiating $g(x) = f(x,\gamma(x))$  with respect to $x$ gives
\begin{equation} \frac{\partial g}{\partial x}(x) = \frac{\partial f}{\partial x}(x,\gamma(x)) + \frac{\partial f}{\partial y}(x,\gamma(x)) = \frac{\partial f}{\partial x}(x,\gamma(x)),\label{eq:chainrule1}\end{equation}
from which we see $\frac{d\gamma}{dx}$ is $\mathcal C^1$ (i.e.\ $\gamma$ is $\mathcal C^2$).  Then differentiating \eqref{eq:chainruleatmin} with respect to $x$ yields
$$0= C + (\frac{d\gamma}{dx})^t D = C + \Gamma^t D$$
giving \eqref{eq:hessiancalcderivativew}.
 Now differentiating \eqref{eq:chainrule1} with respect to $x$ gives
$$\Hess_{x}(g) = B + C \frac{d\gamma}{dx} = B + C\Gamma= B - \Gamma^t D \Gamma.$$
  %Thus $$\Hess_{x_0}g = B - U^t D U.$$   
  So in terms of second order jets
$$ J^2_{x}(g) =  (f(x,\gamma(x)), \frac{\partial f}{\partial x}|_{(x,\gamma(x)}, B - \Gamma^t D\Gamma) = i^*_{\Gamma} J^2_{(x,\gamma(x))}(f)$$
which is \eqref{eq:eq:hessiancalcderivativejetg}. 

 Finally assume $f$ is $F\#\calPos$-subahamonic.  Then since it is also $\mathcal C^2$ we know that $J^2_{(x,\gamma(x_0))}(f) \in (F\#\calPos)_{(\gamma,\gamma(x))}$ and hence (by the definition of  product subequations) $J^2_{x}(g) = i_{\Gamma}^* J^2_{(x,\gamma(x))}(f) \in F_{x}$.  So as $g$ is $\mathcal C^2$ and $x$ is arbitrary, Lemma \ref{lem:fsubharmonicc2} yields $g$ is $F$-subharmonic as claimed.
 \end{proof}

A simple argument with the implicit function theorem shows that if $y\mapsto f(x,y)$ is strictly convex and attains its (necessarily unique) minimum at a point $\gamma(x)$ then the function $x\mapsto \gamma(x)$ is $\mathcal C^1$.   However in the case we will be interested in, the map $y\mapsto f(x,y)$ is convex but not necessarily strictly convex (so any such minimum need not be unique).    If $F$ depends only on the Hessian part, then one can circumvent this problem by approximating $f$ by adding a small multiple of the function $\phi(x,y):= \|y\|^2$.    Such an approximation will be strictly convex in the $y$ direction, and remains $F\#\calPos$-subharmonic as the Hessian of $\phi$ is strictly positive. 

 The next proposition is needed to deal with the possibility of gradient dependence of $F$.   We emphasise that among the hypothesis is that $m$ is equal to 1.

\begin{proposition}\label{prop:hessiancalcII}

 Let $X\subset \mathbb R^{n}$ be open and $F\subset J^2(X)$ be a primitive subequation.  Suppose $$f:\Omega'\to \mathbb R$$ is such that
\begin{enumerate}[(a)]
\item $\Omega'\subset X\times \mathbb R$ is open and has connected fibres.\addtocounter{enumi}{1}
%\item $\Omega'\subset U\times \mathbb R$ for some open $U\subset X$
\item $f$ is $F\#\calPos$-subharmonic.\addtocounter{enumi}{1}
\item $f$ attains its fibrewise minimum strictly in the interior.\addtocounter{enumi}{1}
%\item  $f$ is bounded from below
%\item $f$ is continuous
\item $f$ is $\mathcal C^2$.
%\item For each $a\in \mathbb R$ there is an $M_a$ such that 
%$$ f(x,y)<a \Rightarrow \|y\|<M_a \text{ for all } (x,y)\in \Omega'$$
%\item For each $a\in \mathbb R$ and relatively compact $V\subset \pi(\Omega)$ the set
%$$ \{ (x,y)\in \Omega : f(x,y)<a  \text{ and } x\in V \}$$
%is relatively compact in $\Omega$.

\end{enumerate}
Then the marginal function 
$$ g(x): = \inf_{y\in \Omega_x} f(x,y) \text{ for } x\in \pi(\Omega)$$
is $F$-subharmonic.
\end{proposition}

\begin{proof}
 Fix $x_0\in \pi(\Omega)$.  By (e) there is an $a$ so $g(x_0)<a$ and a neighbourhood $V$ of $x_0$ that is relatively compact in $U$ such that the set
$$ K := \{ (x,y) \in \Omega : f(x,y)<a \text{ and } x\in V\}$$
is relatively compact, and $g<a$ on $V$.
%Fix $M$ so that
%$$(x,y)\in K \Rightarrow |y|<M.$$
Note that for each $x\in V$ the function $y\mapsto f(x,y)$ is convex, and since $\Omega_x$ is connected, we see that $K_x$ is also connected.  Now write in block form
\begin{equation}\label{eq:hessianblockformII} \Hess_{(x,y)} (f) = \left( \begin{array}{cc}\hat{B} &{\hat{C}}  \\ {\hat{C}}^t & {\hat{D}} \\ \end{array}\right)\end{equation}
(so $\hat{B},\hat{C},\hat{D}$ are functions of $(x,y)$ which is dropped from notation).  Given $\alpha,j\in \mathbb N$  we define
\begin{align}
\phi &:= \phi(y) := e^{-\alpha y}\\
%\tilde{\phi}(x) &:= \phi(\gamma(x)) \\
\hat{\Gamma}&:=\hat{\Gamma}(x,y):= -(\hat{D}+j^{-1} \Hess_{y} \phi)^{-1}\hat{C}^t
\end{align}

{\bf Claim I:} For any $\delta>0$ it holds that for all $j\gg \alpha \gg 0$,
\begin{align} |j^{-1} \phi |&\le \delta \\\label{eq:prophessiancalcnonstrictclaimi}
\|j^{-1} \hat{\Gamma}^t \nabla\phi\|&\le \delta\end{align}
uniformly over $(x,y)\in K$.\\

The first statement is immediate as $\phi$ is continuous and $K$ is relatively compact.    For the second statement observe first that the inverse in the definition of $\hat{\Gamma}$ is well-defined as $f$ is convex in the second variable, so $\hat{D}\ge 0$ and $\phi$ is strictly convex so $\Hess_y\phi>0$.  Now
\begin{align*}
 \|j^{-1} \hat{\Gamma}^t \nabla\phi\| &= \| j^{-1} \alpha \hat{\Gamma}^t e^{-\alpha y}\| \\
 &= \| j^{-1} \alpha \hat{C} (\hat{D}+ j^{-1} \alpha^2 e^{-\alpha y})^{-1} e^{-\alpha y}\| \\
&\le \|\hat{C}\| \frac{j^{-1}\alpha e^{-\alpha y}} { \hat{D} + j^{-1}\alpha^2 e^{-\alpha y}}\le \frac{\|\hat{C}\|}{\alpha},
\end{align*}
where the last inequality uses that $\hat{D}\ge 0$.    Now $\|\hat{C}\|$ is bounded uniformly over the relatively compact set $K$,  so Claim I follows.\\

Consider next the function
\begin{align*}
 f_j(x,y)&: = f(x,y) + j^{-1} \phi(y)
 \end{align*}
and set
$$g_j(x): = \inf_{y\in K_x} f_j(x,y).$$

%By relative compactness of $K$ we may pick an $M$ so
%$$ |y|\le M \text{ for all } (x,y)\in K.$$
Fix neighbourhoods $x_0\in U_0\subset V$ and $y_0\in U_1\subset \Omega_{x_0}$ such that $U_0\times U_1\subset K$.\\ % so  $f<a$ on $U_0\times U_1$.     \\

{\bf Claim II: } For $x\in U_0$
$$ g_j(x)\searrow g(x)  \text{ as } j\to \infty.$$

To see this observe that since $f_j$ decreases to $f$ we have
$$ g_j(x) \searrow \inf_{y\in K_x} f(x,y) \text{ as } j\to \infty.$$
On the other hand, as $x\in U_0$ we certainly have $g_j(x)\le a$.  But $f(x,y)\ge a$ for all $y\notin K_x$ and so in fact $\inf_{y\in K_{x}} f(x,y) = \inf_{y\in \Omega_x} f(x,y) = g(x)$.\\

Next recall from Lemma \ref{lem:limitsunderpertubationsofsubequation} the primitive subequation  $F^\delta\subset J^2(X)$ given by
$$ F^{\delta}_x = \{ (r,p,A) : \exists  r',p' \text{ such that } (r',p',A)\in F_x \text{ and } |r-r'|<\delta \text{ and } \|p-p'\|<\delta\}.$$

{\bf Claim III: } For given $\delta>0$ it holds that for all $j\gg \alpha \gg 0$ the function $g_j$ is $F^{\delta}$-subharmonic on $U_0$.\\

Assuming this claim for now, fix such an $\alpha$ and let $j$ tend to infinity to deduce from Claim II that $g$ is $F^\delta$-subharmonic on $U_0$.    Letting $\delta\to 0$ yields that $g$ is in fact $F$-subharmonic on $U_0$ (Lemma \ref{lem:limitsunderpertubationsofsubequation}(3)).  Since $U_0$ is a neighbourhood of an arbitrary point $x_0\in \pi(\Omega)$ we conclude finally that $g$ is $F$-subharmonic on $\pi(\Omega)$ as needed.\\

{\bf Proof of Claim III:}  Let $x\in U_0$.  Since $f$ is convex in $y$ (Lemma \ref{lem:slicesareconvex}) the function $y\mapsto f_j(x,y)$ is strictly convex.   So as $\Omega_x$ is connected, this along with (e) implies $y\mapsto f_j(x,y)$ has a unique minimum which we denote by $\gamma_j(x)$.   Note that by construction $(x,\gamma_j(x))\in K$.  So if $x\in U_0$ then $\gamma_j(x)$ is the unique point that satisfies
$$ \frac{\partial f_j}{\partial y} (x,\gamma_j(x)) =0.$$
As $f_j$ is strictly convex in the $y$ direction, the implicit function theorem implies that $\gamma_j$ is $\mathcal C^1$.    %\begin{equation}\label{eq:prophessIIwbound}
%|\gamma(x)|<M \text{ for all }x\in U_0.
%\end{equation}   
Observe by definition,
$$g_j(x) = f_j(x,\gamma_j(x)) \text{ for } x\in U_0.$$
%and
%$$ \Delta := \Hess_{y_0}(\phi) = \alpha^2 e^{-y}$$
Now
$$ \Hess_{(x,\gamma_j(x))} (f_j) =  \left( \begin{array}{cc} {B}&{C}  \\ {C}^t & {D} +j^{-1}\Hess_{\gamma(x)} \phi \\ \end{array}\right).$$
where $C = C(x) = \hat{C}(x,\gamma_j(x))$ and similarly for $B$ and $D$.  Then Proposition \ref{prop:hessiancalc}(1,2) applies to $f_j|_{K}$, giving
\begin{equation}\label{eq:hessianIIcalcU}
\frac{d\gamma_j}{dx}  = \Gamma_j(x)$$
where
$$\Gamma_j(x) : = \hat{\Gamma}(x,\gamma_j(x)) = -(D+j^{-1}\alpha^2 e^{-\alpha \gamma(x)})^{-1}C^t.\end{equation}
  (we remark that the transpose from equation \eqref{eq:hessiancalcderivativew} has been dropped as $D$ is a $1\times 1$ matrix).  Moreover $g_j$ is $\mathcal C^2$, and its second order jet at a point $x\in U_0$ is given by
\begin{align*}
J^2_{x}(g_j) &= i^*_{\Gamma_j} J^2_{(x,\gamma_j(x))}(f_j) \\
&=i_{\Gamma_j}^* J^2_{(x,\gamma(x))}(f) + j^{-1} i_{\Gamma_j}^* J^2_{(x,\gamma_j(x))}(\phi)\\
%&= i_U^*J^2_{(x,\gamma(x)}(f) + \frac{1}{j}i_U^*J^2_{(x,w(x)}(h)\\
%& =  i_\Gamma^*J^2_{(x,\gamma(x))}(f) + j^{-1} i_\Gamma^*(e^{-\alpha y}, - \alpha e^{-\alpha y}, \alpha^2  e^{-\alpha y}) |_{y=\gamma(x)}\\
& =  i_{\Gamma_j}^*J^2_{(x,\gamma_j(x))}(f) + j^{-1}(\phi(\gamma(x)) , \Gamma_j^t \nabla \phi|_{\gamma_j(x)}, \Gamma_j^t (\Hess_{\gamma_j(x)} \phi)\Gamma_j )\\
& \in F_x + j^{-1}(\phi(\gamma_j(x)) , \Gamma_j^t \nabla \phi_{\gamma_j(x)}, 0 )
\end{align*}
where the last line uses that $f$ is $F\#\calPos$-subharmonic and the Positivity property of $F$.   Then the bounds in Claim I show that for $j\gg \alpha \gg 0$ it holds that $J^2_{x}(g_j) \in F^{\delta}_x$ completing the proof of Claim III and the Proposition.

\end{proof}

%The proof of the previous proposition used the following triviality:
%\begin{lemma}\label{lem:convexityextremumstupid}
%Let $S\subset \mathbb R^n$ be open connected and convex and $h:S\to \mathbb R$ be $\mathcal C^2$ and strictly convex.  Suppose that for some $a\in \mathbb R$ the set $K = \{ y\in S : h(y)<a\}$ is non-empty and relatively compact.  Then $h$ achieves a unique minimum at some point $y_0\in K$.
%\end{lemma}
%\begin{proof}
%Take $y_0$ to be the minimum of $h|_{\overline{K}}$, so as $K$ is non-empty $h(y_0)\le a$.  If $y\notin K$ then $h(y)\ge a\ge h(y_0)$ so $h$ attains is global minimum at $y_0$.  Finally $y_0$ is unique as $S$ is connected and convex.
%\end{proof}

\subsection{Synthesis}

The above analysis in the smooth case, combined with our previous reductions, is enough to prove our first minimum principle.
 
  \begin{theorem}[Minimum Principle in the Real Convex Case]\label{thm:minimumprincipleconvexI}
 Let $F\subset J^2(X)$ be a real primitive subequation such that
 \begin{enumerate}
% \item $\mathbb R_{\le 0} \oplus 0 \oplus 0 \subset F_x$ for all $x\in X$
\item $F$ satisfies the Negativity Property,
 \item $F$ is convex,
 \item $F$ is constant coefficient.
% \item $F$ and $G$ are independent of the gradient part.
 %\item The constant functions on $Y$ are $G$-subharmonic (so $G$ is a convex cone)
 %\item $G\subset \calPos$.
 \end{enumerate}
Then $F$ satisfies the minimum principle.
 \end{theorem}

%Before the proof we start with some simple preparatory results.  

\newcounter{step}
\begin{proof}

This follows by combining Proposition \ref{prop:reductionsmooth} and Proposition \ref{prop:hessiancalcII}.

\end{proof}

\begin{remark}\
When $F = \calPos_{n}$  we have seen in Section \ref{sec:examples} that $F\#\calPos_m = \calPos_{n+m}$.  Then Theorem \ref{thm:minimumprincipleconvexI} becomes the classical statement that if $f(x,y)$ is a function that is convex as $(x,y)$ varies in a convex subset of $\mathbb R^{n\times m}$ whose fibers are connected then the marginal function $g(x) : = \inf_y f(x,y)$ is convex in $x$.
\end{remark}

\section{The Minimum Principle in the Complex Convex Case}\label{sec:minimumcomplexconvex}

We now prove the minimum principle in convex complex case.  The idea of the proof is similar to the real case, but made slightly more complicated as we are considering the complex Hessian and have to make sure that all the terms that we deal with are $\mathbb T$-invariant.

In the following $X\subset \mathbb C^n$ will be open and $F\subset J^{2,\mathbb C}(X)$ a primitive complex subequation.  We will let $z$ be a complex coordinate on $\mathbb C^n$ and $w$ a complex coordinate on $\mathbb C$.  
 
\subsection{Reduction to domains in the complement of zero section}
%By means of notation for $r>0$ set 
%$$ A_r = \{ w\in \mathbb C : |w| >  r\}$$

We first consider a reduction that allows us to work away from the zero section of the second variable (i.e. to work insude $X\times \mathbb C^*$).   This is needed to ensure that when we make a smooth mollification we can retain the property of being $\mathbb T$-invariant.

\begin{proposition}\label{prop:reductionCstar}
Let $F\subset J^{2,\mathbb C}(X)$ be a complex primitive subequation that is constant coefficient and has the Negativity Property.     Suppose for any $$ f:\Omega'\to \mathbb R$$
such that
\begin{enumerate}[(A)]
\item $\Omega'$ is a $\mathbb T$-invariant open subset of $X\times \mathbb C$ with connected fibres,
\item There exists an $F$-subharmonic function on $\pi(\Omega')$ that is bounded from below,
\item $f$ is $F\#_{\mathbb C}\calPos^{\mathbb C}$-subharmonic and $\mathbb T$-invariant,
\item $f$ is bounded from below,
\item $f$ attains its fiberwise minimum strictly in the interior,
%\item $f$ is continuous
\item $\Omega'\subset X\times \mathbb C^*$,
\end{enumerate}
the marginal function $$g(z) = \inf_{w\in \Omega'_z} f(z,w)$$ is $F$-subharmonic on $\pi(\Omega')$.     Then the minimum principle holds for $F$.
\end{proposition}

\begin{proof}
Let $f:\Omega'\to \mathbb R$ be such that (A--D) hold and also

\begin{enumerate}[(E')]%\setcounter{enumi}{5}
\item $f$ is relatively exhaustive.
\end{enumerate}
Set $g(z) = \inf_{w\in \Omega_z'} f(z,w)$ which we will show is $F$-subharmonic on $\pi(\Omega')$.  Then Proposition \ref{prop:reductionlemmarelativelyexhausting} applies to give that the minimum principle holds for $F$.

Fixing $z_0\in \pi(\Omega')$ it is sufficient to prove $g$ is $F$-subharmonic in some neighbourhood $V$ of $z_0$. As $\Omega'_{z_0}$ is $\mathbb T$-invariant and connected we can write
$$ \Omega'_{z_0} = \{w\in \mathbb C: r <  |w|  <R\}$$
for some $0\le r<R\le \infty$.   The proof then splits into two cases.\\
   
\noindent {\bf Case 1: } $r=0$.\\

We are assuming $f$ is bounded from below on $\Omega$, so say $f\ge c$.  Then $(z_0,0)\in \Omega$, so for sufficiently small  neighbourhoods $V$ of $z_0$ we have $(z,0)\in \Omega$ for all $z\in V$.  Shrinking $V$ if necessary we may assume $V\Subset \pi(\Omega')$, and there exists a bounded $F$-subharmonic function on $V$.  In fact shrinking $V$ further if necessary we can fix real numbers $0<R<R<R''$ and $C$ such that for all $z\in V$,
 \begin{enumerate}[(I)]
 \item if $0\le |w|<R''$ then $(z,w)\in \Omega$.
 \item if $|w|<R$  then $f(z,w)<C$.
 \item if $R'<|w|<R''$ then $f(z,w)>2C-c$.
 \end{enumerate}
 (This is possible by using openness and upper-semicontinuity of $f$ to obtain $R$ and $C$ in such a way that $\{(z\in V, |w|<r\}$ is relatively compact in $\Omega'$,  and then using that $f$ is relatively exhaustive to obtain $R'$ and $R''$).  Set
 $$ \lambda: = \frac{R}{R'} <1.$$
 
 {\bf Claim: } For all $j$ sufficiently large there exist $h_j:\Omega_j \to \mathbb R$ such that
 \begin{enumerate}[(i)]
 \item $\pi(\Omega_j)=V$.
 \item $h_j$ has properties (A--F).
 \item For all $z\in V$
 \begin{equation}\label{eq:infsandwithch} \inf_{w\in \Omega_z, |w|>\lambda^{j+1}}   f(z,w) \le  \inf_{w\in \Omega_{j,z}} h_j(z,w) \le  \inf_{w\in \Omega_z, |w|>\lambda^j} f(z,w).\end{equation}
\end{enumerate}

Given this claim for now, we use the hypothesis of the Proposition to conclude that $k_j(z): = \inf_{z\in \Omega_{j,z}} h_j(z,w)$ is $F$-subharmonic on $V$.  But \eqref{eq:infsandwithch} implies also that $k_j$ decreases pointwise to $g(z)=\inf_{w\in \Omega_z} f(z,w)$, which is thus also $F$-subharmonic.\\

For the claim, choose $j$ large enough to $\lambda^{j}<R$.  Define
$$ \Omega_j : = \{ (z,w)\in \Omega' : z\in V \text{ and } |w|> (R/R'') \lambda^j \}$$
and set 
$$ h_j = \max\{ f(z,w),   f(z, R\lambda^j/w) +c-C\}.$$
Observe first that if $(z,w)\in \Omega_j$ then $R\lambda^j/|w| < R''$ so $(z,w)\in \Omega$ and thus $h_j$ is well defined on $\Omega_j$.  Moreover if $z\in V$ and $|w|>\lambda^j$ then $f(z,R\lambda^j/w)+c-C<c<f(z,w)$ so \begin{equation}
(z,w)\in \Omega_j \text{ and }|w|>\lambda^j \Rightarrow h_j(z,w) = f(z,w).\label{eq:zone2}
\end{equation}

On the other hand if $z\in V$ and $(R/R'')\lambda^j < |w|< \lambda^{j+1}$ then $R\lambda^j/|w|>R'$ so using (III), $f(z,R\lambda^j/w) + c-C > C >f(z,w)$ (the last inequality follows from (II) since we are assuming $j$ is large enough so $\lambda^{j+1}<R$).  Thus we have
\begin{equation}
(z,w)\in \Omega_j \text{ and }  |w|< \lambda^{j+1} \Rightarrow h_j(z,w) >C. \label{eq:zone1}
\end{equation}

Now as $\lambda^j<R<R''$ item (I) implies $\pi(\Omega_j) =V$ so (i) holds.  Proposition \ref{lem:compositionbiholomorphism} implies that the function $(z,w)\mapsto f(z,R\lambda^j/w)$ is $F\#\calPos$-subharmonic, hence so is $h_j$.  From this properties (A-D) and (F) for the function $h_j$ are immediate, and property (E) (namely that $h_j$ attains its minimum strictly in the interior) follows from \eqref{eq:zone1} and \eqref{eq:zone2} giving (ii).

From \eqref{eq:zone2} if $z\in V$ then
$$ \inf_{w\in \Omega_{j,z}} h_j(z,w) =  \inf_{w\in \Omega'_{z}, |w|>\lambda^{j+1}}   h_j(z,w).$$
Then observing that $\Omega' \cap \{ |w|>\lambda^j\} \subset \Omega_{j}\cap \{ |w|>\lambda^{j+1}\}$ we have
$$ \inf_{w\in \Omega_{j,z}} h_j(z,w) =  \inf_{w\in \Omega_{j,z}, |w|>\lambda^{j+1}}   h_j(z,w) \le \inf_{w\in \Omega'_{z}, |w|>\lambda^{j}} f(z,w).$$ 
On the other hand $h_j\ge f$ everywhere so
\begin{align*}
 \inf_{w\in \Omega_{j,z}} h_j(z,w) &=  \inf_{w\in \Omega_{j,z}, |w|>\lambda^{j+1}}   h_j(z,w) \ge  \inf_{w\in \Omega_{j,z}, |w|>\lambda^{j+1}}  f(z,w)\\&\ge  \inf_{w\in \Omega'_{z}, |w|>\lambda^{j+1}}  f(z,w)\end{align*}
where the last inequality uses $\Omega_j\subset \Omega'$.  Thus we have (iii) giving the claim, and completing the proof in the case $r=0$.\\

\noindent {\bf Case 2: } $r>0$.\\

Fix $a>g(z_0)$.  As $f$ is assumed to be relatively exhaustive, there is relatively compact neighbourhood $V$ of $z_0$ such that the set $ K = \{ (z,w)\in \Omega' : f(z,w)<a \text{ and } z\in V\}$ is relatively compact in $\Omega'$.   Hence (just because it is relatively compact) we can arrange, by shrinking $V$ is necessary, that there is an $r'>r$ such that
$$ K \subset  X\times \{ (z,w) : |w|>r'\}.$$
Furthermore there is no loss in assuming that there exists an $F$-subharmonic function on $V$ that is bounded from below.

Now set
 $$j:= f|_{\Omega' \cap (V\times \{ (z,w) : |w|>r\})}.$$
Then by construction,
\begin{enumerate}[(i)]
\item $j$  has properties (A--F)
\item The marginal function of $j$ equals the marginal function of $f$ on $V$, i.e.
$$\inf_{w} j(z,w) = \inf_{w\in \Omega'_z} f(z,w) = g(z) \text{ for } z\in V.$$
\end{enumerate}

Given this, our hypothesis apply to the $j'$ giving that $g$ is $F$-subharmonic on $V$, which completes the case $r=0$ and the proof of the Proposition is finished.

\end{proof}

\subsection{Reduction to the smooth complex case}

\begin{proposition} \label{prop:reductionsemiconvexcomplex} Let $F\subset J^{2,\mathbb C}(X)$ be a complex constant-coefficient convex primitive subequation that has the Negativity Property.    Then Proposition \ref{prop:reductionCstar} continues to hold we assume in addition that
\begin{enumerate}[(G)]%\setcounter{enumi}{5}
%\item $\Omega'\subset \pi^{-1}(U)$
%\item $f$ is $F\#\calPos$-subharmonic 
%\item $f$ attains its fiberwise minimum in the interior
\item $f$ is semiconvex.
\end{enumerate}
\end{proposition}
\begin{proof}
The proof is the same as that of Proposition \ref{prop:reductionsemiconvex}, only using Corollary \ref{cor:prop:approxsemiTinvariant} instead of Proposition \ref{prop:approxsemi} to ensure that the approximating semiconvex functions $f_j$ are $\mathbb T$-invariant.
\end{proof}

\begin{proposition} \label{prop:reductionsmoothcomplex} Let $F\subset J^{2,\mathbb C}(X)$ be a complex constant-coefficient convex primitive subequation that has the Negativity Property.    Then Proposition \ref{prop:reductionsemiconvexcomplex}  continues to hold if we assume in addition that
\begin{enumerate}[(H)]%\setcounter{enumi}{5}
%\item $\Omega'\subset \pi^{-1}(U)$
%\item $f$ is $F\#\calPos$-subharmonic 
%\item $f$ attains its fiberwise minimum in the interior
\item $f$ is smooth.
\end{enumerate}
\end{proposition}

\begin{proof}
The proof is the same as the proofs of Proposition \ref{prop:reductionsmooth}, only using Proposition \ref{prop:smoothapproximation_complex} instead of Proposition \ref{prop:smoothapproximation} to ensure that the approximating smooth functions $h_j$ are $\mathbb T$-invariant (and we observe here that we are using in a crucial way condition (F) that says $\Omega'\subset X\times \mathbb C^*$).
\end{proof}
 
\subsection{The minimum principle in the complex smooth case}

\begin{proposition}\label{prop:hessiancalccomplex}
 %Let $X\subset \mathbb C^{n}$ be open and $F\subset J^{2,\mathbb C}(X)$ be a primitive complex subequation.  
 Let $\Omega\subset X\times \mathbb C^{*}$ be open and $\mathbb T^m$-invariant, and assume $f:\Omega\to \mathbb R$ is such that
 \begin{enumerate}[(1)]
%\item   $f$ is $F\#G$-subharmonic
\item   $f$ is $\mathcal{C}^2$ and $\mathbb T^m$-invariant and strictly pseudoconvex in the second variable.
\item There exists a real $\mathcal C^1$ function $\gamma:X\to \mathbb R^m\subset \mathbb C^m$ such that
$$ g(z) :=\inf_{w\in \Omega_z} f(z,w) =  f(z,\gamma(z)) \text{ for } z\in \pi(\Omega).$$
%\item   For each fixed $x\in \pi(\Omega)$ the function $y\mapsto f(x,y)$ is strictly convex and attains a unique minimum at some point $w(x)$.
\end{enumerate}
Then 
\begin{enumerate}[(a)]
\item The derivative of $\gamma$ is given by
\begin{equation}\Gamma:=\frac{d\gamma}{dz}  = -\frac{1}{2{D}} C^*\label{eq:hessiancalcderivativewcomplex}\end{equation}
where
\begin{equation}\label{eq:hessianblockformcomplex} \Hess^{\mathbb C}_{(z,\gamma(z)} (f) = 2\left( \begin{array}{cc} B&C  \\ C^* & D \\ \end{array}\right)\end{equation}
is the complex Hessian matrix of $f$ at $(z,\gamma(z))$ in block form (by which we mean $B_{ij} = \frac{\partial^2 f}{\partial z_i \partial  \overline{z}_j}|_{(z,\gamma(z))}$, $C_{i} = \frac{\partial^2 f}{\partial z_i \partial  \overline{w}}|_{(z,\gamma(z))}$ and $D = \frac{\partial^2 f}{\partial w \partial  \overline{w}}|_{(z,\gamma(z))}$).    

%Furthermore
%$$\frac{d\gamma}{d\overline{z}} = \overline{\frac{d\gamma}{d{z}}} = blah.$$
\item The marginal function $g$ is $\mathcal C^2$, and its second order complex jet at a point $x\in \pi(\Omega)$ is given by 
\begin{equation}\label{eq:hessiancalcderivativejetgcomplex}
J^{2,\mathbb C}_{z}(g) = i^*_{2\Gamma } J^{2,\mathbb C}_{(z,\gamma(z))}(f)
\end{equation}
where $i^*$ is as defined in \eqref{eq:defofistar}. 
\item If additionally $f$ is $F\#_{\mathbb C}\calPos^{\mathbb C}$-subharmonic on $\Omega$ then $g$ is $F$-subharmonic on $\pi(\Omega)$.
\end{enumerate}
\end{proposition}
\begin{proof}
As $\gamma(z)$ is a minimum of the function $w\mapsto f(z,w)$ we have
\begin{align}\label{eq:chainruleatmin_complex1} f_{\overline{w}} (z,\gamma(z)) \equiv 0\text{ and} \\
f_{{w}} (z,\gamma(z)) \equiv 0. \nonumber \end{align}
Differentiating \eqref{eq:chainruleatmin_complex1} with respect to $z$,
$$f_{\overline{w} z} (z,\gamma(z)) + f_{w\overline{w}} (z,\gamma(z)) \gamma_z(z) + f_{ww}(z,\gamma(z)) \gamma_z =0$$
where we have used that $\gamma$ is real so $(\overline{\gamma})_z = \gamma_z$.    Now $f$ is $\mathbb T^m$-invariant, so \eqref{eq:chainruleatmin_complex1} implies
$$ f_{ww} (z,\gamma(z)) = f_{w\overline{w}}(z,\gamma(z))$$
(this can be seen, for instance, by using polar coordinates).   Thus in fact
$$f_{\overline{w} z} (z,\gamma(z)) + 2 f_{w\overline{w}} (z,\gamma(z)) \gamma_z(z)=0$$
giving \eqref{eq:hessiancalcderivativewcomplex}.

For the second statement, by definition
$$ g(z) = f(z,\gamma(z))$$
so differentiating with respect to $z$ gives
\begin{align*} g_z &= f_z(z,\gamma(z)) + f_w(z,\gamma(z)) \gamma_z + f_{\overline{w}}(z,\gamma(z)) \gamma_z\\
&= f_z(z,\gamma(z)). \label{eq:chainrule1complex}\end{align*}
Differentiating with respect to $\overline{z}$ gives
$$g_{z\overline{z}}=f_{z\overline{z}}(z,\gamma(z))  +  f_{zw}(z,\gamma(z))  \gamma_{\overline{z}} +  f_{z\overline{w}}(z,\gamma(z)) \overline{\gamma_z}.$$
After some manipulation, in terms of second order jets we deduce
$$ J^{2,\mathbb C}_{z}(g) =  (f(z,\gamma(z)), 2 \frac{\partial f}{\partial \overline{z}}|_{(z,\gamma(z))}, 2B - 8\Gamma^* D\Gamma)) = i^*_{2\Gamma} J^2_{(z,\gamma(z)}(f)$$
which is \eqref{eq:hessiancalcderivativejetgcomplex}. 

The final statement follows from this (precisely as in the proof of Proposition \ref{prop:hessiancalc}).
 \end{proof}
 
 The following perturbation argument deals with the fact that $w\mapsto f(z,w)$ may not attain a (unique) minimum.  We stress that we use in an essential way that $\Omega\subset X\times \mathbb C^*$.
 
 \begin{proposition}\label{prop:hessiancalcIIcomplex}
 Let $X\subset \mathbb C^{n}$ be open and $F\subset J^{2,\mathbb C}(X)$ be a complex primitive subequation.   Suppose
$$f:\Omega'\to \mathbb R$$ is such that
\begin{enumerate}[(A)]
\item $\Omega'\subset X\times \mathbb C$ is open, $\mathbb T$-invariant and has connected fibres.\addtocounter{enumi}{1}
\item $f$ is $F\#_{\mathbb C}\calPos^{\mathbb C}$-subharmonic and $\mathbb T$-invariant.
\addtocounter{enumi}{1}
\item $f$ attains is fibrewise minimum strictly in the interior.
\item $\Omega\subset X\times \mathbb C^*$.\addtocounter{enumi}{1}
\item $f$ is $\mathcal C^2$.
\end{enumerate}
Then the marginal function 
$$ g(z): = \inf_{w\in \Omega_z} f(z,w) \text{ for } z\in \pi(\Omega)$$
is $F$-subharmonic.
\end{proposition}
\begin{proof}
This is similar to the proof of Proposition \ref{prop:hessiancalcII}, but made slightly more complicated due to the presence of the complex variable.   For completeness we give the entire argument here.

 Fix $z_0\in \pi(\Omega)$, pick $a$ so $g(z_0)<a$ and let $w_0$ be such that $f(z_0,w_0)<a$.  Fix also an open neighbourhood $x_0\in V\Subset \pi(\Omega)$ and let 
$$ K := \{ (z,w) \in \Omega : f(z,w)<a \text{ and } z\in V\}$$
which by hypothesis is relatively compact.   Since we are assuming $\Omega\subset X\times \mathbb C^*$, the set $K$ is bounded away from $w=0$.  So we may fix a large real $M$ so
\begin{equation}(z,w)\in K \Rightarrow M^{-1}<\|w\|< M.\label{eq:defM}\end{equation}
Note that $K$ is also $\mathbb T$-invariant, so shrinking $V$ is necessary we may assume that each $K_z$ is a non-empty annulus in $\mathbb C^*$ (so in particular each $K_z$ is connected).
%Fix $M$ so that
%$$(x,y)\in K \Rightarrow |y|<M.$$

Write in block form
\begin{equation}\label{eq:hessianblockformIII} \Hess^{\mathbb C}_{(x,y)} (f) = 2 \left( \begin{array}{cc}\hat{B} &{\hat{C}}  \\ {\hat{C}}^* & {\hat{D}} \\ \end{array}\right)\end{equation}
(so $\hat{B},\hat{C},\hat{D}$ are functions of $(x,y)$ which is dropped from notation).  Given $\alpha,j\in \mathbb N$  we define
\begin{align}
\phi &:= \phi(w) := e^{\alpha |w|^2}\\
%\tilde{\phi}(x) &:= \phi(\gamma(x)) \\
\hat{\Gamma}&:=\hat{\Gamma}_j(z,w):= -\frac{1}{2(\hat{D}+j^{-1} \phi_{w\overline{w}})}\hat{C}^*.
\end{align}

{\bf Claim I:} For any $\delta>0$ it holds that for all $j\gg \alpha \gg 0$,
\begin{align} |j^{-1} \phi |&\le \delta \\\label{eq:prophessiancalcnonstrictclaimicomplex}
\|4 j^{-1}  \hat{\Gamma}^*  \phi_{\overline{w}}  \|&\le \delta\end{align}
uniformly over $(z,w)\in K$.\\

The first statement is immediate from the definition of $\phi$ as $K$ is relatively compact.    For the second statement observe first that the inverse in the definition of $\hat{\Gamma}$ is well-defined as $f$ is $\calPos^{\mathbb C}$-subharmonic in the second variable, so $\hat{D}\ge 0$ and $\phi$ is strictly plurisubharmonic so $\Hess^{\mathbb C}_y\phi$ is strictly positive.  Now
\begin{align}
 \phi_{\overline w} &= \alpha e^{\alpha |w|^2}{w}\\
 \phi_{w\overline{w}} &=    e^{\alpha |w|^2}(\alpha^2 |w|^2 + \alpha)
\end{align}
so
\begin{align*}
\| 4j^{-1}  \hat{\Gamma}^* \phi_{\overline{w}} \| \le \frac{2j^{-1}\alpha e^{\alpha|w|^2}|w| \|\hat{C}^*\|}{\hat{D} + j^{-1}  e^{\alpha |w|^2}(\alpha^2 |w|^2 + \alpha)}
\le  \frac{  2\|\hat{C}^*\| }{\alpha |w|}
\end{align*}
as $\hat{D}\ge 0$.   Using \eqref{eq:defM} this is bounded above by $\delta$ as long as $\alpha$ is sufficiently large as $\|\hat{C}\|$ is bounded uniformly over the relatively compact set $K$.  Thus  Claim I follows.\\

Consider next the function
\begin{align*}
 f_j(z,w)&: = f(z,w) + j^{-1} \phi(w)
 \end{align*}
and set
$$g_j(z): = \inf_{y\in K_z} f_j(z,w).$$

%By relative compactness of $K$ we may pick an $M$ so
%$$ |y|\le M \text{ for all } (x,y)\in K.$$
Fix neighbourhoods $x_0\in U_0\subset V$ and $y_0\in U_1\subset \Omega_{x_0}$ such that $U_0\times U_1\subset K$.\\ % so  $f<a$ on $U_0\times U_1$.     \\

{\bf Claim II: } For $x\in U_0$
$$ g_j(x)\searrow g(x)  \text{ as } j\to \infty.$$

This is proved exactly as for Proposition \ref{prop:hessiancalcIIcomplex} so is not repeated.\\

{\bf Claim III: } For given $\delta>0$ it holds that for all $j\gg \alpha \gg 0$ the function $g_j$ is $F^{\delta}$-subharmonic on $U_0$.\\

Assuming this claim for now, fix such an $\alpha$ and let $j$ tend to infinity to deduce from Claim II that $g$ is $F^\delta$-subharmonic on $U_0$.    Letting $\delta\to 0$ yields that $g$ is in fact $F$-subharmonic on $U_0$ (Lemma \ref{lem:limitsunderpertubationsofsubequation}(3)).  Since $U_0$ is a neighbourhood of an arbitrary point $z_0\in \pi(\Omega)$ we conclude finally that $g$ is $F$-subharmonic on $\pi(\Omega)$ as needed.\\

{\bf Proof of Claim III:}  Let $x\in U_0$.  Since $f$ is plurisubharmonic in $y$ (Lemma \ref{lem:slicesareconvex}) the function $f_j$ is strictly plurisubharmonic and exhaustive.  As it is also $\mathbb T$-invariant this implies  $w\mapsto f_j(z,w)$ has a unique minimum which we denote by $\gamma_j(z)$, which for each $z$ is the unique point satisfying
$$ \frac{\partial f}{\partial w} (z,\gamma(z))=0.$$
Again using that $f$ is $\mathbb T$-invariant, the implicit function theorem implies that $\gamma_j$ is $\mathcal C^1$.   Note that by construction $(z,\gamma_j(z))\in K$. 
%\begin{equation}\label{eq:prophessIIwbound}
%|\gamma(x)|<M \text{ for all }x\in U_0.
%\end{equation}   

Thus by definition,
$$g_j(z) = f_j(z,\gamma(z)) \text{ for } z\in U_0.$$
%and
%$$ \Delta := \Hess_{y_0}(\phi) = \alpha^2 e^{-y}$$
Now
$$ \Hess^{\mathbb C}_{(z,\gamma(z))} (f_j) = 2 \left( \begin{array}{cc} {B}&2{C}  \\ {C}^* & {D} +j^{-1} \phi_{w\overline{w}}\\ \end{array}\right),$$
where $C = C(z) = \hat{C}(z,\gamma_j(z))$ and similarly for $B$ and $D$.  Then Proposition \ref{prop:hessiancalccomplex}(1,2) applies to $f_j|_{K}$, giving
\begin{equation}\label{eq:hessianIIcalcU_complex}
\Gamma_j(z):= \frac{d\gamma_j}{dz}  = \hat{\Gamma}_j(z,\gamma(z))\end{equation}
where $\hat{\Gamma}_j$ is as above  (we remark that the transpose from equation \eqref{eq:hessiancalcderivativew} has been dropped as $m=1$).  Moreover $g_j$ is $\mathcal C^2$, and its second order jet at a point $z\in U_0$ is given by

\begin{align*}
J^2_{z}(g_j) &= i^*_{2\Gamma_j} J^2_{(z,\gamma_j(z)}(f_j) \\
&=i_{2\Gamma_j^*} J^2_{(z,\gamma_j(z))}(f) + j^{-1} i_{2\Gamma_j^*} J^2_{(z,\gamma_j(z))}(\phi)\\
& =  i_{2\Gamma_j^*}J^2_{(z,\gamma_j(z))}(f) + j^{-1}(\phi(\gamma_j(z)) , 4 \Gamma_j^*  \phi_{\overline{w}}|_{\gamma_j(z)}, 8 \Gamma_j^* \phi_{w\overline{w}}|_{\gamma_j(z)}\Gamma_j )\\
& \in F_z + j^{-1}(\phi(\gamma_j(z)) , 4 \Gamma_j^*  \phi_{\overline{w}}|_{\gamma_j(z)}, 0 )
\end{align*}
where the last line uses that $f$ is $F\#_{\mathbb C}\calPos^{\mathbb C}$-subharmonic and the Positivity property of $F$.   Then \eqref{eq:prophessiancalcnonstrictclaimi} applies to show that for $j\gg \alpha \gg 0$ it holds that $J^2_{z}(g_j) \in F^{\delta}_z$ completing the proof of Claim III.\\

From Claim II and Claim III the Proposition follows, exactly as for Proposition \ref{prop:hessiancalcIIcomplex}.
\end{proof}

\subsection{Synthesis}

 \begin{theorem}[Minimum Principle in the Complex Convex Case]\label{thm:minimumprinciplecomplexconvexI}
 Let $X\subset \mathbb C^n$ be open and $F\subset J_{2,\mathbb C}(X)$ be a complex primitive subequation such that
 \begin{enumerate}
% \item $\mathbb R_{\le 0} \oplus 0 \oplus 0 \subset F_x$ for all $x\in X$
\item $F$ has the Negativity Property,
 \item $F$ is convex,
 \item $F$ is constant coefficient.
% \item $F$ and $G$ are independent of the gradient part.
 %\item The constant functions on $Y$ are $G$-subharmonic (so $G$ is a convex cone)
 %\item $G\subset \calPos$.
 \end{enumerate}
Then $F$ satisfies the minimum principle.
 \end{theorem}
 
 \begin{proof}
This follows by combining Proposition \ref{prop:reductionsmoothcomplex} and Proposition \ref{prop:hessiancalcIIcomplex}.

 \end{proof}
 
 \begin{remark}\
When $F = \calPos^{\mathbb C}_{\mathbb C^n}$  we have seen in Example \ref{sec:exampleI} that $F\#\calPos_{\mathbb C^m} = \calPos_{\mathbb C^{n+m}}$.  Then Theorem \ref{thm:minimumprinciplecomplexconvexI} is precisely the Kiselman minimum principle.
\end{remark}
\begin{corollary}\label{cor:minimumimaginary}
With $F$ as in Theorem \ref{thm:minimumprinciplecomplexconvexI}, suppose that $\Omega\subset X\times \mathbb C$ is an $F\#_{\mathbb C}\calPos^{\mathbb C}$-pseudoconvex domain and $f:\Omega\to \mathbb R\cup \{-\infty\}$ is $F\#_{\mathbb C}\calPos^{\mathbb C}$-subharmonic.  Assume that
\begin{enumerate}
\item $\Omega$ is independent of the imaginary part of the second variable.  That is
$$ (z,w)\in \Omega \text{ and } \Im(w) = \Im(w') \Rightarrow (z,w')\in \Omega.$$
\item $\Omega$ has connected fibres
\item $f$ is independent of the imaginary part of the second variable.  That is
$$ (z,w)\in \Omega \text{ and } \Im(w) = \Im(w') \Rightarrow f(z,w) = f(z,w').$$
\item $\Omega$ admits an exhaustive continuous $F\#_{\mathbb C}\calPos^{\mathbb C}$-subharmonic function $u$ that is independent of the imaginary part of the second variable.
\end{enumerate}
Then $\pi(\Omega)$ is $F$-pseudoconvex and the marginal function
$$ g(z) : = \inf_{z\in \Omega_z} f(z,w)$$
is $F$-subharmonic.
\end{corollary}
\begin{proof}
Let $\phi(z,w) = (z,e^{w})$ and set $\Omega' = \phi(\Omega)$.    Then $u$ descends to an exhaustive continuous function $u'$ on $\Omega'$, which by Lemma \ref{lem:compositionbiholomorphism} is $F\#_{\mathbb C}\calPos^{\mathbb C}$-subharmonic.  Hence $\Omega'$ is $F\#_{\mathbb C}\calPos^{\mathbb C}$-pseudoconvex.  Moreover $f$ descends to an $F\#_{\mathbb C}\calPos^{\mathbb C}$-subharmonic function $f'$ on $\Omega'$.   But clearly $\pi(\Omega') = \pi(\Omega)$ and the marginal function of $f'$ is the same as the marginal function of $f$, so the minimum principle for $F$ applied to $(\Omega',f')$ gives the result we want.
\end{proof}

\appendix

\section{F-subharmonic Functions}\label{appendix:subequations}

\subsection{Types of Subequations}

\begin{definition}
Let $F\subset J^2(X)$.  
\begin{enumerate}
\item We say $F$ is \emph{constant coefficient} if $F_x$ is independent of $x$, i.e.\
$$ (x,r,p,A) \in F_x \Leftrightarrow (x',r,p,A)\in F_{x'} \text{ for all }x,x',r,p,A.$$
\item We say $F$ is \emph{independent of the gradient part} (or \emph{gradient-independent})  if each $F_x$ is independent of $p$, i.e. 
$$ (r,p,A)\in F_x \Leftrightarrow (r,p',A) \in F_x \text{ for all }  x,r,p,p',A.$$
\item We say $F$ \emph{depends only on the Hessian part} if each $F_x$ is independent of $(r,p)$,  i.e.
$$ (r,p,A) \in F_x \Leftrightarrow (r',p',A)\in F_{x} \text{ for all }x,r,r',p,p',A.$$
\end{enumerate}
\end{definition}

%\subsection{The $GL_n$-action}
\begin{definition}[$G$-Invariance]The group $GL_n(\mathbb R)$ acts on $J^2(X)$  by
$$ g^*(x,r,p,A):= (x,r,g^t p, g^t Ag) \text{ for } g\in GL_n(\mathbb R).$$
If $G$ is a subgroup of $GL_n(\mathbb R)$ we say $F\subset J^2(X) $ is \emph{$G$-invariant} if $g^*\alpha \in F$ for all $\alpha\in F$ and all $g\in G$.
\end{definition}

\begin{remark}
Our action of $GL_n$ comes from thinking of the jet space using the cotangent space to $X$, and is different in convention to that of \cite{HL_Dirichletdualitymanifolds}.
\end{remark}

\begin{comment}
\begin{proposition}\label{prop:quadratic}
Assume that $F\subset J^2$ is closed and let $f$ be upper-semicontinuous.  Then $f$ is not $F$-subharmonic if and only if there exists a point $x_0$ and $\epsilon>0$ and a quadratic function
$$q(y):=  r + p^t(x-x_0) +\frac{1}{2}(x-x_0)^t A (x-x_0)$$
such that
\begin{enumerate}
\item 
    $f(x) - q(x)  \le - \epsilon |x-x_0|^2$
\item $f(x_0) - q(x_0)=0$
\item  $J^2_{x_0}(q)\notin F$
\end{enumerate}
\end{proposition}
\begin{proof}
\cite[Lemma 2.4]{HLRiemannian}
\end{proof}

\end{comment}

\subsection{Complex Subequations}\label{sec:complexsubequations}
Set
$$ \mathbb J  = \left(\begin{array}{cc} 0& -\Id_{n} \\ \Id_n & 0\end{array}  \right)\in M_{2n\times 2n}(\mathbb R).$$
If $A\in M_{2n\times 2n}(\mathbb R)$ commutes with $\mathbb J$ then making the standard identification $\mathbb C\simeq \mathbb R^2$ we think of $A$ as a complex matrix $\hat{A} \in M_{n\times n}(\mathbb C)$.  Explicitly if in block form
$$ A = \left(\begin{array}{cc} a& c \\ b & d \end{array} \right)$$
where $a,b,c,d\in M_{n\times n}(\mathbb R)$ then $A$ commutes with $\mathbb J$ if and only if $a=d$ and $b=-c$,  in which case
$$\hat{A}: = a+ib \in M_{n\times n}(\mathbb C).$$
Observe $\widehat{AB} = \hat{A} \hat{B}$ and $\widehat{A^t} = \hat{A}^*$.

Let $\Herm_n$ be the set of hermitian $n\times n$ complex matrices, and 
%$\Pos^{\mathbb C}_n\subset \Herm(\mathbb C^n)$ 
$$\Pos_n^{\mathbb C} := \{ \hat{A} \in \Herm_n : v^* \hat{A} v \ge 0 \text{ for all } v\in \mathbb C^n\}$$
the subset of semipositive hermitian matrices.  From the above it easy to check that if $A\mathbb J = \mathbb JA$ then
\begin{align*}
A\in \Sym^2_{2n} &\Longleftrightarrow \hat{A}\in \Herm(\mathbb C^n)\text{ and }\\
A \in \Pos_{2n}&\Longleftrightarrow \hat{A}\in \Pos^{\mathbb C}_{n}.
\end{align*}
Now for any $A\in M_{2n\times 2n}(\mathbb R)$ the matrix
$$ A_{\mathbb C} : = \frac{1}{2} ( A - \mathbb J A \mathbb J)$$
commutes with $\mathbb J$ and thus we may think of $A_{\mathbb C}$ as an element of $M_{n\times n}(\mathbb C)$.  Observe if $A$ is symmetric then $A_{\mathbb C}$ is hermitian.

\begin{definition}
Let $X\subset \mathbb R^{2n}\simeq \mathbb C^n$ be open.  We say $F\subset  J^2(X)$ is a \emph{complex subequation} if $(x,r,v,A)\in F$ if and only if $(x,r,v,A_{\mathbb C})\in F$.  
\end{definition}
So by abuse of notation if $F$ is a complex subequation we may equivalently consider it as a subset
$$ F\subset J^{2,\mathbb C}(X) := X\times \mathbb R \times \mathbb C^n \times \Herm(\mathbb C^n) =: X\times J^{2,\mathbb C}_n$$
without any loss of information.  The group $GL_n(\mathbb C)$ acts on $J^{2,\mathbb C}(X)$ by
$$ g^*(x,r,p,A) = (x,r,g^* p, g^* Ag).$$
Observe also if $F$ is complex, then having the Positivity property \eqref{eq:positivity} is equivalent to  
$$(x,r,p,A)\in F \Longrightarrow (x,r,p,A+P)\in F 
\text{ for all }P\in \Pos_n^{\mathbb C}.$$ 
\begin{example}
Let 
$$\calPos^{\mathbb C}_X : = X\times \mathbb R\times \mathbb C^n \times \Pos^{\mathbb C}_n$$
which is a convex complex subequation.  We will write $\calPos^{\mathbb C}$ for $\calPos^{\mathbb C}_X$ when $X$ is clear from context.
\end{example}

\begin{example}[Convex and Plurisubharmonic]\label{example:convexity:appendix}
Recall $\mathcal P_X = X\times \mathbb R\times \mathbb R^n\times \Pos_n$.  Then $\mathcal P_X(X)$ consists of locally convex functions on $X$ \cite[Example 14.2]{HL_Dirichletdualitymanifolds}.  Similarly if $X\subset \mathbb C^n$ is open then $\mathcal P^{\mathbb C}_X(X)$ consists of the plurisubharmonic functions on $X$ \cite[p63]{HL_Dirichletdualitymanifolds}.
\end{example}

\subsection{Basic properties of $F$-subharmonic functions}

The following lists some of the basic limit properties satisfied by $F$-subharmonic functions (under very mild assumptions on $F$).
\begin{proposition}\label{prop:basicproperties}
Let $F\subset J^2(X)$ be closed.  Then 
\begin{enumerate}
\item (Maximum Property) If $f,g\in F(X)$ then $\max\{f,g\}\in F(X)$.
\item (Decreasing Sequences) If $f_j$ is decreasing sequence of functions in $F(X)$ (so $f_{j+1}\le f_j$ over $X$) then $f:=\lim_j f_j$ is in $F(X)$.
\item (Uniform limits) If $f_j$ is a sequence of functions on $F(X)$ that converge locally uniformly to $f$ then $f\in F(X)$.
\item (Families locally bounded above) Suppose $\mathcal F\subset F(X)$ is a family of $F$-subharmonic functions locally uniformally bounded from above.  Then the upper-semicontinuous regularisation of the supremum
$$ f:= {\sup}^*_{f\in \mathcal F} f$$
is in $F(X)$.
\item If $F$ is constant coefficient and $f$ is $F$-subharmonic on $X$ and $x_0\in \mathbb R^{n}$ is fixed,  then the function $x\mapsto f(x-x_0)$ is $F$-subharmonic on $X-x_0$.
\end{enumerate}
\end{proposition}
\begin{proof}
See \cite[Theorem 2.6]{HL_Dirichletdualitymanifolds} for (1-4).  Item (5) is immediate from the definition.
\end{proof}

\begin{lemma}[Limits under perturbations of subequations]\label{lem:limitsunderpertubationsofsubequation}
Let $X$ be open and $F\subset J^2(X)$ be a primitive subequation.  For $\delta>0$ let $F^\delta\subset J^2(X)$ be defined by
$$ F^{\delta} = \{ (x,r,p,A) : \exists r',p' \text{ with } (x,r',p',A)\in F \text{ and } |r-r'|\le \delta \text{ and } \|p-p'\|\le \delta\}.$$
Then
\begin{enumerate}
\item  $F^{\delta}$ is a primitive subequation.  
\item If $F$ satisfies the Negativity property then so does $F^\delta$.
\item $\bigcap_{\delta>0} (F^\delta(X))=F(X)$.\end{enumerate}
\end{lemma}

\begin{proof}
That $F^{\delta}$ has the Positivity property is immediate from the definition, and $F^{\delta}$ is closed as $F$ is closed giving (1).   Statement  (2) is also immediate from the definition. Finally using $F$ is closed, $\bigcap_{\delta>0} F^{\delta}_x = F_x$, and thus
$\bigcap_{\delta>0} (F^{\delta}(X)) = F(X).$
\end{proof}

\subsection{$F$-subharmonicity in terms of second order jets}
It is useful to understand the property of being $F$-subharmonic in terms of second order jets.  To do so we first discuss what it means to be twice differentiable at a point.   Again let $X\subset \mathbb R^n$ be open.

\begin{definition}[Twice differentiability at a point]
We say that $f:X\to \mathbb R$ is \emph{twice differentiable} at $x_0\in X$ if there exists a $p\in \mathbb R^n$ and an $L\in \Sym_n^2$ such that for all $\epsilon>0$ there is a $\delta>0$ such that for $\|x-x_0\|<\delta$  we have
\begin{equation}\label{eq:twicediff} |f(x) - f(x_0) - p.(x-x_0) - \frac{1}{2}  (x-x_0)^tL (x-x_0) | \le \epsilon \|x-x_0\|^2.
\end{equation}
\end{definition}

When $f$ is twice differentiable at $x_0$ then the $p,L$ in \eqref{eq:twicediff} are unique, and moreover in this case $f$ is differentiable at $x_0$ and 
$$p = \nabla f|_{x_0}= \left(  \begin{array}{c} \frac{\partial f}{\partial x_1} \\ \frac{\partial f}{\partial x_2} \\ \vdots \\ \frac{\partial f}{\partial x_n} \end{array}\right)|_{x_0}\in \mathbb R^n.$$
When $f$ is twice differentiable at $x_0$ we shall refer to $L$ as the \emph{Hessian} of $f$ at $x_0$ and denote it by $\Hess(f)|_{x_0}$.  Of course, by Taylor's Theorem,  when $f$ is $\mathcal C^{2}$ in a neighbourhood of $x_0$ then $\Hess_{x}(f)$ is the matrix with entries
$$(\Hess(f)_{x_0} )_{ij}: = \frac{\partial^2 f}{\partial x_i\partial x_j}|_{x_0}.$$

\begin{definition}[Second order jet]
Suppose that $f: X\to \mathbb R$ is twice differentiable at $x_0$.    We denote the \emph{second order jet} of $f$ at $x_0$ by
\begin{equation}J^2_{x_0}(f):= (f(x_0), \nabla f|_{x_0}, \Hess(f)|_{x_0}) \in J^2_{n} = \mathbb R\times \mathbb R^n\times \Sym_n^2.\label{eq:secondorderjet}
\end{equation}
%The \emph{reduced second order jet} is
%\begin{equation}J^2_{x_0,red}(f):= (d f|_{x_0}, \Hess(f)|_{x_0}) \in \mathbb R\times \mathbb R^n \times \Sym^2_n.\label{eq:secondorderjetreduced}
%\end{equation}
\end{definition}

The importance of the Positivity property is made apparent by the following that shows that $F$-subharmonicity behaves as expected for sufficiently smooth functions.

\begin{lemma}\label{lem:fsubharmonicc2}
Let $F\subset J^2(X)$ satisfy the Positivity assumption \eqref{eq:positivity} and suppose $f:X\to \mathbb R$ is $\mathcal C^2$.  Then $f\in F(X)$ if and only if $J^2_{x}(f)\in F_x$ for all $x\in X$.
\end{lemma}
\begin{proof}
The reader may easily verify this, or consult \cite[Equation 2.4 and Proposition 2.3]{HL_Dirichletdualitymanifolds}.  
\end{proof}

The definition $F$-subharmonicity given above says that at any upper-contact point $x$, with upper-second order jet $(p,A)$,  the quadratic function
$$ y\mapsto f(y) + p.(x-y) + \frac{1}{2}   (y-x)^tA (y-x)$$
has second-order jet lying in $F_x$.  The next statement says that this is equivalent to the more classical  ``viscosity definition".  Given an upper-semicontinuous $f$ we say that $\phi$ is a $\mathcal{C}^2$-\emph{test function touching $f$ from above at $x_0$} if $\phi\in\mathcal C^2$ in a neighbourhood of $x_0$ with $\phi\ge f$ on this neighbourhood and $\phi(x_0) = f(x_0)$.

\begin{lemma}[Viscosity definition of $F$-subharmonicity]\label{lem:viscositydefinition}
 An upper-semicontinuous $f:X\to \mathbb R\cup \{-\infty\}$ is in $F(X)$ if and only if for all $x_0\in X$ and test-functions $\phi$ touching $f$ from above at $x_0$ it holds that $J^2_{x_0}(\phi)\in F_{x_0}.$
\end{lemma}
\begin{proof}
See \cite[Lemma 2.4]{HL_Dirichletdualitymanifolds}.
\end{proof}

It takes some work to understand how $F$-subharmonicity interacts with linearity in the space of functions.      However when $F$ is constant-coefficient and convex the following is true:

\begin{proposition}\label{prop:convexcombintation}[Convex combinations of $F$-subharmonic functions]
Let $F$ be a constant coefficient convex primitive subequation.  Then any convex combination of $F$-subharmonic functions is again $F$-sub\-harmonic.
\end{proposition}
\begin{proof}
This is implied by \cite[Theorem 5.1 ]{HL_AE}  (apply the cited theorem to $F_x:=\lambda H_x$ and $G_x:= (1-\lambda)H_x$ for a given $\lambda\in [0,1]$)
\end{proof}

\section{Associativity of products}\label{sec:associativityproductsubequations}
We prove Proposition \ref{prop:associativityofproducts} which states that if $X_i\subset \mathbb R^{n_i}$ are open and $F_i\subset J^2(X_i)$ for $i=1,2,3$ then
$$ (F_1\#F_2)\#F_3 = F_1\#(F_2\#F_3).$$

Let $x,y,z$ be coordinates on $\mathbb R^{n_1},\mathbb R^{n_2},\mathbb R^{n_3}$ respectively.     We will consider certain linear mappings
\begin{align*}
\Gamma&:\mathbb R^{n_1}\to \mathbb R^{n_2+n_3}\\
\Phi&:\mathbb R^{n_1+n_2}\to \mathbb R^{n_3}\\
\Psi&:\mathbb R^{n_1} \to \mathbb R^{n_2}\\
\Upsilon&:\mathbb R^{n_2}\to \mathbb R^{n_3}
\end{align*}
and write
$$ \Phi(x,y) = \Phi_1(x) + \Phi_2(y)$$
where $\Phi_i:\mathbb R^{n_i} \to \mathbb R^{n_3}$ is linear.  Recall that $\iota_{\Gamma}:\mathbb R^{n_1} \to \mathbb R^{n_1+n_2+n_3}$ is $\iota_{\Gamma}(x) = (x,\Gamma(x))$ and similarly for $\iota_\Phi:\mathbb R^{n_1+n_2}\to \mathbb R^{n_1+n_2+n_3}$ and $\iota_{\Psi}:\mathbb R^{n_1}\to \mathbb R^{n_1+n_2}$. 

\begin{lemma}\label{lem:factoriota}
Suppose 
\begin{equation}\label{eq:GammaintermsofPhiPsi}\Gamma = (\Psi, \Phi_1 + \Phi_2\circ\Psi). \end{equation}
Then 
$$\iota_{\Gamma} = \iota_{\Phi}\circ\iota_{\Psi}$$
\end{lemma}
\begin{proof}
\begin{align*} \iota_{\Phi}(\iota_{\Psi}(x)) &= \iota_{\Phi}(x,\Psi(x)) = (x,\Psi(x),\Phi(x,\Psi(x))) \\&= (x,\Psi(x),\Phi_1(x) + \Phi_2\circ\Psi(x)) = \iota_{\Gamma}(x).\end{align*}
\end{proof}

Now set
$$\begin{array}{ll}
j_2: \mathbb R^{n_2} \to \mathbb R^{n_1+n_2} &\text{  } j_2(y) = (0,y)\\
j_3: \mathbb R^{n_3} \to \mathbb R^{n_1+n_2+n_3} &\text{  } j_3(z) = (0,0,z)\\
j_{23}:\mathbb R^{n_2+ n_3} \to \mathbb R^{n_1+n_2+n_3} &\text{  } j_{23}(y,z) = (0,y,z)\\
k:\mathbb R^{n_3} \to \mathbb R^{n_2+n_3} &\text{  } k(z) = (0,z).
\end{array}$$
Fix $(x,y,z)\in X_1\times X_2\times X_3$.  By definition of the product subequation we know $\alpha \in (F_1\#(F_2\#F_3))_{(x,y,z)}$  if and only if
\begin{equation}\label{eq:associative1}
\forall \Gamma \text{ we have }\iota_\Gamma^*\alpha \in (F_1)_{x} \text{ and } j_{23}^*\alpha \in (F_2\#F_3)_{(y,z)}.
\end{equation}
Observe for every $\Gamma$ there is a pair $(\Phi,\Psi)$ such that \eqref{eq:GammaintermsofPhiPsi} holds.   Thus by Lemma \ref{lem:factoriota}, condition \eqref{eq:associative1} is equivalent to
\begin{equation}\label{eq:associative2}
\forall \Psi,\Phi \text{ we have }\iota^*_\Psi \iota^*_\Phi \alpha \in (F_1)_{x} \text{ and } j_{23}^*\alpha \in (F_2\#F_3)_{(y,z)}.
\end{equation}
Using the definition of $F_2\#F_3$, condition \eqref{eq:associative2} is in turn equivalent to
\begin{equation}\label{eq:associative3}
\forall \Psi,\Phi,\Upsilon \text{ we have }\iota^*_\Psi \iota^*_\Phi \alpha \in (F_1)_{x} \text{ and } \iota_{\Upsilon}^* j_{23}^*\alpha \in (F_2)_{y} \text{ and }k^*j_{23}^*\alpha \in (F_3)_{z}.
\end{equation}
Now $j_{23}\circ k=j_3$, and a simple check yields $j_{23}\circ \iota_{\Phi_2} = \iota_{\Phi}\circ j_2$.  Thus \eqref{eq:associative3} is equivalent to
\begin{equation}\label{eq:associative4}
\forall \Psi,\Phi \text{ we have }\iota^*_\Psi \iota^*_\Phi \alpha \in (F_1)_{x} \text{ and } j_{2}^*\iota_{\Phi}^*\alpha \in (F_2)_{y} \text{ and }j_3^*\alpha \in (F_3)_{z}.
\end{equation}
So from the definition of $(F_1\#F_2)$, condition \eqref{eq:associative4} is equivalent to
\begin{equation}\label{eq:associative5}
\forall \Phi \text{ we have }\iota^*_\Phi \alpha \in (F_1\#F_2)_{x} \text{ and }j_3^*\alpha \in (F_3)_{z} 
\end{equation}
which, by definition, is equivalent to $\alpha\in ((F_1\#F_2)\#F_3)_{(x,y,z)}$.

\section{Products of gradient-independent Subequations}\label{appendixB}

Recall we say that a subequation $F\subset J^2(X)$ has Property (P\textsuperscript{++}) if the following holds.  For all $x\in X$ and all $\epsilon>0$ there exists a $\delta>0$ such that

\begin{equation}(x,r,p,A)\in F_x \Rightarrow (x',r-\epsilon,p, A+\epsilon \Id)\in F_{x'} \text{ for all } \|x'-x\|<\delta\ \tag{P\textsuperscript{++}}\label{eq:propertyH}
\end{equation}
or said another way,
\begin{equation}F_x + (x'-x,0,-\epsilon, \epsilon \Id) \subset F_{x'} 
\text{ for all } \|x'-x\|<\delta. \tag{P\textsuperscript{++}}\label{eq:propertyH:repeat}
\end{equation}

\begin{lemma}\label{lem:inclusioninterior2}
Assume that $F$ and $G$ have property \eqref{eq:propertyH} and are independent of the gradient part.   Then $H: = F\#G$ is a subequation
\end{lemma}

\begin{proof}
We have already seen in Lemma \ref{lem:productsofsubequations} that $H$ is closed, and satisfies the Positivity and Negativity properties \eqref{eq:positivity} and \eqref{eq:negativity}.  It remains to prove the Topological property \eqref{eq:topological} which we break up into a number of pieces.  Since $F,G$ are independent of the gradient part, so is $H$, and thus the only non-trivial part of the topological property is to show \cite[Section 4.8]{HL_Dirichletdualitymanifolds}
$$ \Int(H_{(x,y)}) =(\Int H)_{(x,y)}$$
The fact that $\Int(H_{(x,y)}) \subset (\Int H)_{(x,y)}$ is obvious, so the task is to prove the other inclusion.

Let $$\alpha \in \Int( H_{(x,y)}),$$ so there exists a $\delta_1>0$ such that
\begin{equation} \| \hat{\alpha} - \alpha \| <\delta _1 \text{ and } \hat{\alpha} \in J^2(X\times Y)|_{(x,y)}  \Rightarrow \hat{\alpha}\in H_{(x,y)}.\label{eq:int1}\end{equation}
By hypothesis there is a $\delta_2>0$ such that

\begin{equation}\label{eq:intopen1} F_x +(x'-x, -\frac{\delta_1}{4}, 0,\frac{\delta_1}{4} \Id) \subset F_{x'} \text{ for } \|x-x'\|<\delta_2\end{equation}

\begin{equation}\label{eq:intopen2} G_y +(x'-x, -\frac{\delta_1}{4}, 0,\frac{\delta_1}{4} \Id) \subset G_{y'} \text{ for } \|y-y'\|<\delta_2.\end{equation}

Set $\delta = \min\{ \delta_1/2,\delta_2/2\}$ and pick any $\alpha'\in J^2(X\times Y)$ with 
$$ \|\alpha'-\alpha\|<\delta.$$
We will show that $\alpha' \in H$.

Denote the space coordinate of $\alpha'$ by $(x',y')$, so $\alpha'\in J^2(X\times Y)|_{(x',y')}$.  Thus we certainly have $\|x'-x\|<\delta<\delta_2$ and $\|y-y'\|<\delta_2$.  Define
$$\hat{\alpha} : = \alpha' + ((x-x',y-y'),-\frac{\delta_1}{2},0, -\frac{\delta_1}{2} \Id_{n+m}).$$
Then $\hat{\alpha} \in J^2(X\times Y)_{(x,y)}$ and
$$ \|\hat{\alpha} - \alpha\| \le \| \alpha' - \alpha \| + \| \hat{\alpha} - \alpha'\|  < \delta_1.$$
Thus \eqref{eq:int1} applies, so $\hat{\alpha}\in H_{(x,y)}$ which means
$$ j^* \hat{\alpha} \in G_{y} \text{ and } i_U^* \hat{\alpha}\in G_{x} \text{ for all } U.$$
Now using \eqref{eq:intopen2}.
$$ j^*\alpha' = j^*\hat{\alpha} + (y'-y,\frac{\delta_1}{2},0, \frac{\delta_1}{2} \Id_{m}) \in  G_y +  (y'-y,\frac{\delta_1}{2},0, \frac{\delta_1}{2} \Id_{m}) \subset G_{y'}.$$

Similarly using the Positivitiy property of $F$ and \eqref{eq:intopen1}
\begin{align*}
 i_U^*\alpha' &= i_U^*\hat{\alpha} + (x-x',\frac{\delta_1}{2},0,\frac{\delta_1}{2} \Id_{n}  + \frac{\delta_1}{2} U^tU) \\
 &\subset  i_U^*\hat{\alpha} + (x-x',\frac{\delta_1}{2},0,\frac{\delta_1}{2} \Id_{n})\\
 &\subset F_{x'}.
\end{align*}
Thus $\alpha'\in H_{(x',y')}$.  As this holds for all such $\alpha'$ we conclude $\alpha \in \Int(H)$ completing the proof.
\end{proof}

\subsection{The complex case}
Let $X\subset  \mathbb R^{2n}\simeq \mathbb C^n$ be open.   If $f:X\to \mathbb R$ is twice differentiable at a point $z\in X $ its \emph{complex Hessian} is
$$ \Hess^{\mathbb C}_z(f) =  \frac{1}{2} ( \Hess(f) - \mathbb J \Hess_x(f) \mathbb J) \in \Herm(\mathbb C^n).$$  When $f$ is sufficiently smooth we have
$$ ( \Hess^{\mathbb C}_z(f))_{jk} = 2 \frac{\partial^2 f}{\partial z_j\partial\overline{z}_k}|_z$$
where, as usual,
$$ \frac{\partial}{\partial z_j} = \frac{1}{2} \left( \frac{\partial}{\partial x_j} - i \frac{\partial}{\partial y_j}\right) \text{ for } z_j = x_j + i y_j.$$
In terms of the gradient, under the identification $\mathbb R^{2n}\simeq\mathbb C^n$ we have
$$\nabla f|_z = \left(\begin{array}{c} \frac{\partial f}{\partial x}|_z \\ \frac{\partial f}{\partial {y}}|_z \end{array}\right)  = 2\frac{\partial f}{\partial \overline{z}}|_z.$$
\begin{definition}[Complex $2$-jet]
The complex $2$-jet of $f$ at $z\in X$ is
$$ J^{2,\mathbb C}_{z} (f)  := ( f(z),  2\frac{\partial f}{\partial \overline{z}}|_z,\Hess^{\mathbb C}_z(f)) \in J^{2,\mathbb C}_{z} = \mathbb R\times \mathbb C^n\times \Herm_n.$$
\end{definition}
So if $F\subset J^2(X)$ is complex then
$$ J^2_{z} (f)\in F_z \Longleftrightarrow J^{2,\mathbb C}_{z} (f)\in F_z.$$

\let\oldaddcontentsline\addcontentsline% Store \addcontentsline
\renewcommand{\addcontentsline}[3]{}% Make \addcontentsline a no-op
\bibliographystyle{plain}
\bibliography{minimumprinciple}{}
\let\addcontentsline\oldaddcontentsline% Restore \addcontentsline

\medskip
\small{
\noindent {\sc Julius Ross,  Mathematics Statistics and Computer Science, University of Illinois at Chicago, Chicago  IL, USA\\ julius@math.uic.edu}\medskip

\noindent{\sc David Witt Nystr\"om, 
Department of Mathematical Sciences, Chalmers University of Technology and the University of Gothenburg, Sweden \\ wittnyst@chalmers.se, danspolitik@gmail.com}

}

\end{document}